\documentclass[10pt]{amsart}
\usepackage[utf8]{inputenc}
\usepackage{color}

\usepackage[makeroom]{cancel}

\usepackage{amsthm}
\usepackage{amssymb}
\usepackage{amsmath}
\usepackage{pdfpages}
\usepackage{booktabs}
\usepackage{graphics}
\usepackage{tikz,lipsum,lmodern,ulem}
\usepackage{tikz-cd}
\usepackage{sidecap}
\sidecaptionvpos{figure}{c}
\usepackage[most]{tcolorbox}
\colorlet{shadecolor}{gray!40}
\usepackage[bookmarksnumbered=true]{hyperref} 
\hypersetup{
     colorlinks = true,
     linkcolor = blue,
     anchorcolor = blue,
     citecolor = blue,
     filecolor = blue,
     urlcolor = blue
     }
\usetikzlibrary{patterns}

\theoremstyle{plain}

\newtheorem*{maintheorem}{Theorem A}
\newtheorem{theorem}{Theorem}[section]
\newtheorem{conjecture}[theorem]{Conjecture}
\newtheorem{proposition}[theorem]{Proposition}
\newtheorem{corollary}[theorem]{Corollary}
\newtheorem{lemma}[theorem]{Lemma}

\theoremstyle{definition}

\newtheorem{definition}[theorem]{Definition}
\newtheorem{example}[theorem]{Example}

\def\FF{{\mathbb F}}

\def\kb{{\mathbf k}}
\def\xb{{\mathbf x}}
\def\ini{\mathrm{in}}

\theoremstyle{remark}

\newtheorem{remark}[theorem]{Remark}

\title[Gr\"obner bases, resolutions, and the Lefschetz properties]{Gr\"obner bases, resolutions, and the Lefschetz properties for powers of a general linear form in the squarefree algebra}

\author[Jonsson Kling, Lundqvist, Mohammadi, Orth, S\'aenz-de-Cabez\'on ]{Filip Jonsson Kling, Samuel Lundqvist, Fatemeh Mohammadi, Matthias Orth, and Eduardo S\'aenz-de-Cabez\'on}

\address{Filip Jonsson Kling, Department of Mathematics, Stockholm University, 
Sweden}
\email{filip.jonsson.kling@math.su.se}

\address{Samuel Lundqvist, Department of Mathematics, Stockholm University, 
Sweden}
\email{samuel@math.su.se}

\address{{\footnotesize Fatemeh Mohammadi, Departments of Computer Science and Mathematics, 
KU Leuven, Belgium}}
\email{fatemeh.mohammadi@kuleuven.be}

\address{Matthias Orth, Institute of Mathematics, University of Kassel,
Germany}
\curraddr{Department of Mathematics, 
KU Leuven, Belgium}
\email{ matthias.orth@kuleuven.be}

\address{Eduardo S\'aenz-de-Cabez\'on, Departamento de Matemáticas y Computación, Universidad de La Rioja, Spain}
\email{eduardo.saenz-de-cabezon@unirioja.es}

\begin{document}

\begin{abstract}
For the almost complete intersection ideals $(x_1^2, \dots, x_n^2, (x_1 + \cdots + x_n)^k)$, we compute their reduced Gröbner basis for any term ordering, revealing a combinatorial structure linked to lattice paths, elementary symmetric polynomials, and Catalan numbers. Using this structure, we classify the weak Lefschetz property for these ideals. Additionally, we provide a new proof of the well-known result that the squarefree algebra satisfies the strong Lefschetz property. Finally, we compute the Betti numbers of the initial ideals and construct a minimal free resolution using a Mayer-Vietoris tree approach.
\end{abstract}

\maketitle

{\hypersetup{linkcolor=black}
\setcounter{tocdepth}{2}
{\tableofcontents}}

\section{Introduction}
Let $\mathbf{k}$ be a field of characteristic zero, let $R = \mathbf{k}[x_1, \ldots, x_n]$ be the polynomial ring in $n$ variables, and let $I = (x_1^{a_1}, \ldots, x_n^{a_n})$. Stanley \cite{Stanley} and Watanabe \cite{WATANABE1989194} independently showed that the Artinian monomial complete intersection $R/I$ has the strong Lefschetz property. This result has been central to the subsequent and extensive work on Lefschetz properties in commutative algebra. Showing that $R/I$ has the SLP is equivalent to showing that the Hilbert series of
$$S = R/(x_1^{a_1}, \ldots, x_n^{a_n}, (x_1 + \cdots + x_n)^{a_{n+1}})$$
equals
$$\left[\frac{\prod_{i=1}^{n+1}(1 - t^{a_i})}{(1 - t)^n}\right],$$
where the brackets denote truncation at the first non-positive coefficient.

This establishes a natural connection to the Fr\"oberg conjecture \cite{Froberg}, which states that the Hilbert series for the quotient of $R$ by the ideal generated by $m$ general forms of degrees $a_1, \ldots, a_m$ is given by
$$ \left[ \frac{\prod_{i=1}^{m} (1 - t^{a_i})}{(1 - t)^n} \right]. $$

An algebra that is the quotient of an ideal generated by forms, not necessarily general, of degrees $a_1, \ldots, a_m$, whose Hilbert series agrees with the one conjectured by Fr\"oberg, is called \emph{thin}. It is well known that the existence of a thin algebra corresponding to a given degree sequence provides a proof of the conjecture for that sequence. Thus, Stanley's and Watanabe's result shows that $S$ is thin, and that the Fr\"oberg conjecture is true for $n+1$ general forms.

Many proofs of the fact that $S$ is thin have appeared in the literature. Hara and Watanabe \cite{Watanabe_Boolean} used a reduction argument to the quadratic case. They show that if $\mathbf{k}[x_1, \ldots, x_m]/(x_1^2, \ldots, x_m^2, (x_1 + \cdots + x_m)^{a_{m+1}})$ is thin, where $m$ is the socle degree $a_1 + \cdots + a_n - n$ of $S$, then $S$ is also thin. Equivalently, the strong Lefschetz property of squarefree algebras in any number of variables implies the strong Lefschetz property for all monomial complete intersections in any number of variables.

The quadratic case turns out to be a key building block when studying the weak Lefschetz property for almost complete intersection algebras of the form $R/(x_1^d, \ldots, x_n^d, (x_1 + \cdots + x_n)^d)$. A result by Sturmfels and Xu \cite{SturmfelsXu} for the case $d = 2$ was the key tool used by Boij and the second author \cite{boijlundqvist} to classify the weak Lefschetz property for any $d$. However, the argument here differs from the one used by Hara and Watanabe. In this case, it is the \emph{failure} of the weak Lefschetz property for $d = 2$ that is shown to induce failure for higher values of $d$, except for a few specific values of $n$.

The ideal $(x_1^{a_1}, \ldots, x_n^{a_n}, (x_1 + \cdots + x_n)^{a_{n+1}})$, which may initially appear quite specialized, is in fact a general object. After a linear change of coordinates, we obtain the isomorphism
$$ S \cong R/(\ell_1^{a_1}, \ldots, \ell_{n+1}^{a_{n+1}}), $$
where the $\ell_i$ are general linear forms. This reveals a deeper geometric connection, particularly in the context of the interpolation of general fat points via Macaulay's inverse system, as described by Emsalem and Iarrobino \cite{emsalemiarrobino}.
In light of this, it is natural to further explore the structure of the thin algebras generated by $n+1$ powers of general linear forms. 

Our main result is that we explicitly determine the Gr\"obner bases for
$$ I_{n,k} = (x_1^2, \ldots, x_n^2, (x_1 + \cdots + x_n)^k). $$
As is well known, Gr\"obner bases are a key tool in computer algebra systems for performing computations. However, they are also known for their lack of respect for symmetry, strong dependence on the monomial ordering, and often challenging theoretical description. In contrast, the Gr\"obner basis we determine exhibits a different behavior.
First, all Gr\"obner basis elements, except for the squares, are elementary symmetric polynomials in a subset of the variables, and in particular, they have  $0,1$-coefficients.
Second, the sequence formed by the number of Gr\"obner basis elements of degree greater than 2 corresponds to the $(k-1)$-fold convolution of Catalan numbers.
Third, the Gr\"obner basis is independent of the monomial ordering once we fix the ordering $x_1 \succ \cdots \succ x_n$, and it is also invariant under permutations of the first $k$ variables and the action of transpositions of the form $(k+2i - 1,k+ 2i)$ for $i \geq 1$.

\medskip
We summarize our main result in the following theorem. A more general version of this theorem is stated as Theorem~\ref{thm:GBasis} in Section~2.
\begin{maintheorem}\label{thm:A} 
A reduced Gröbner basis of $I_{n,k} = (x_1^2, \dots, x_n^2, (x_1 + \cdots + x_n)^k)$ for $k\geq 2$ 
is given by
\begin{equation*}\label{eq:g elements}
G_{n,k}=\{ x_1^{2},\ldots,x_n^2 \}\cup \bigcup_{d=k}^{k+\lfloor (n-k)/2\rfloor}\left\{ g_{A,n,k}\mid A\in\mathcal{A}, |A|=d \right\},
    \end{equation*}
where $\mathcal{A}$ is the family of inclusion-minimal members of the family $\mathcal{A}'$ of subsets $A \subseteq \{1, \ldots, n\}$ satisfying $\max(A) = 2|A| - k$. For $|A|=d$,
\begin{equation}\label{eq:g_{A,n,k}-intro}
    g_{A,n,k} = e_d(x_{i_1},\dots, x_{i_{n-d+k}})
\end{equation}
is the elementary symmetric polynomial of degree $d$ in the variables indexed by the set $\{i_1,\dots, i_{n-d+k}\}=A\cup \{2d-k+1,\dots, n\}$ with leading term $x_A=\prod_{a\in A}x_a$. 

Moreover, for fixed $k$ the sequence of cardinalities $|\{A\in\mathcal{A}: |A|=d\}|$ is a $(k-1)$-fold convolution of the sequence of Catalan numbers.
\end{maintheorem}
The case when $k=1$ is also treated, see Remark \ref{rem:k=1} for details. The case when $k>n$ is trivial because then $I_{n,k}=(x_1^2,\ldots,x_n^2)$ is monomial; nevertheless, the formula for $G_{n,k}$ includes this case.

Notice that we can present 
$I_{n,k}$ as $(x_1^2,\ldots,x_n^2,e_k(x_1,\ldots,x_n))$, which gives a connection to a work by Haglund, Rhoades, and Shimozono \cite{HAGLUND2018851} on the Delta conjecture \cite{haglund2018delta}. They determine the lexicographic Gr\"obner bases for the class of ideals $(x_1^k,\ldots,x_n^k,e_n(x_1,\ldots,x_n), \ldots, e_{n-k+1}(x_1,\ldots,x_n))$, and surprisingly show that its elements can be described in terms of Demazure characters, which provides an overlap with our class in the trivial cases $I_{n,n-1}$ and $I_{n,n}$.

From the Gr\"obner basis, we can determine the initial ideal, and use its structure to provide a new proof of the well-known fact that the squarefree algebra satisfies the strong Lefschetz property. We then extend our analysis and present our main result on the Lefschetz properties which is that 
$$R/I_{n,k} \text{ has the WLP if and only if }
\begin{cases}
k\geq \frac{n-3}{2} & \text{ for } n \text{ odd,}\\
k\geq \frac{n}{2} & \text{ for } n \text{ even.}\\
\end{cases}$$

We prove this result by combining Gr\"obner deformation techniques from Conca \cite{Conca} and Wiebe \cite{Wiebe} with a thorough investigation of relations between pairs of powers of general linear forms in the squarefree algebra. As a corollary, we also obtain a minor result related to the Fr\"oberg conjecture.

It is well known that the Lefschetz properties often serve as a bridge between different areas of mathematics, with well-established connections to commutative algebra, algebraic geometry, combinatorics, representation theory, and algebraic topology. In this spirit, one can view the present paper as contributing a link to computer algebra.

Finally, we study the minimal free resolution of the initial ideal of $I_{n,k}$.~We~show that its squarefree part is strongly squarefree stable, allowing us to apply results by Gasharov, Hibi, and Peeva \cite{GasharovHibiPeeva2002} to obtain closed formulas for the Betti numbers. 
The initial ideal itself is a combination of squarefree and powers ideals, for which Murai \cite{Murai2008} has provided recursive formulas for their Betti numbers.~To~explicitly compute the Betti numbers, we construct a minimal free resolution using Mayer-Vietoris trees, as described by the fifth author \cite{Saenz-De-Cabezon2009MVT}.~Our main result regarding the homological properties of the initial ideal of $I_{n,k}$ 
is that the sequence of extremal Betti numbers is a $(k-1)$-fold convolution of the Catalan number sequence.

Some of the ideas developed in this paper were later extended in \cite{klmo}.

\begin{remark} During the finalization of the paper, we were informed that Booth, Singh, and Vraciu in \cite{booth2024weak} independently have provided a description of the initial ideal of $(x_1^d,\ldots,x_n^d,(x_1+\cdots+x_n)^d)$ for $d=2,3$, using different techniques. Their work overlaps with ours regarding the initial ideal for $I_{n,2}$. 
They also find relations for establishing failure of the weak Lefschetz property (WLP) due to injectivity for $I_{n,2}$ for certain $n$, which are similar to but distinct from special cases of the relations we obtain for establishing failure of the WLP for $I_{n,k}$ for certain $n$ and $k$.
\end{remark}

\section{The Gr\"obner basis} \label{sec:description} 
\noindent Recall that a Gr\"obner basis for an ideal $I$ with respect to a term ordering $\prec$ is a set $\{g_1,\ldots, g_r\}$ such that the leading monomials of the $g_i$ generate the initial ideal of $I$, denoted $\ini(I)$, with respect to $\prec$. A Gr\"obner basis is \emph{reduced} if it is in bijection to the minimal generating set of $\ini(I)$, all its elements have $1$ as leading coefficient, and the tail of every polynomial in it consists of monomials that are not in the initial ideal.

We fix our notation used throughout this note. We define $[n] = \{1, \ldots, n\}$. For a subset $S = \{i_1, \ldots, i_s\} \subseteq [n]$, we denote by $x_S$ the monomial $x_{i_1} \cdots x_{i_s}$ in the polynomial ring $\kb[x_1, \ldots, x_n]$.

\subsection{Constructing the Gr\"obner basis elements}
In this section, we will prove that the polynomials $g_{A,n,k}$ do belong to the ideal $I_{n,k}$, establishing one part of Theorem~\ref{thm:GBasis}. An alternative description for the polynomials $g_{A,n,k}$ will also be deduced in Theorem~\ref{thm:f_description}. We will split the proofs into several lemmas. 

\begin{definition}\label{def:SFP}
    The \emph{squarefree part} of a polynomial $f\in \mathbf{k}[x_1,\ldots,x_n]$ is 
    its normal form with respect to the monomial ideal $(x_1^2,\ldots,x_n^2)$. We write $\mathrm{SFP}(f)$ for the squarefree part of $f$. For a subset of polynomials $P\subseteq \mathbf{k}[x_1,\ldots,x_n]$ we write $\mathrm{SFP}(P)=\{ \mathrm{SFP}(f)\mid f\in P \}$.
\end{definition}
In other words, the squarefree part of a polynomial is the part obtained by removing all terms that contain a square. For example, 
$$\mathrm{SFP}((x_1+\cdots+x_5)^2) = 2 (x_1 x_2 + x_1 x_3 + \cdots + x_4 x_5) = 2 e_2(x_1,\ldots,x_5),$$ where $e_2$ is the elementary symmetric polynomial of degree two.

Since $(x_1^2,\ldots,x_n^2)\subset I_{n,k}$ for all $n$ and $k$, the crucial part of the analysis of the homogeneous ideals $I_{n,k}$ is the degreewise description of their squarefree parts $\mathrm{SFP}((I_{n,k})_{(d)})$.

\begin{lemma}\label{lem:SFPofPolys_and_I_n_k}
    Let $f\in \mathbf{k}[x_1,\ldots,x_n]$ be a polynomial. Then $f-\mathrm{SFP}(f)\in I_{n,k}$. In particular, $f\in I_{n,k}$ if and only if $\mathrm{SFP}(f)\in I_{n,k}$.
\end{lemma}
\begin{proof}
    Observing that $(x_1^2,\ldots,x_n^2)\subseteq I_{n,k}$, we only need to apply Definition~\ref{def:SFP}.
\end{proof}

\begin{definition}\label{def:g_S_n_k}
    Let $S\subseteq [n]$ be an index set. We write 
\begin{equation}\label{eq:f_{S,n,k}}
       f_{S,n,k}=\frac{1}{k!}\mathrm{SFP}(x_S(x_1+\cdots +x_n)^k).
    \end{equation}
\end{definition}

Using the polynomials $f_{S,n,k}$, we obtain generating systems of the squarefree parts in single degrees of the ideals $I_{n,k}$.

\begin{lemma}\label{lem:GenSysDegreewiseSFP}
    \begin{enumerate}
        \item 
        Let $S\subseteq [n]
        $. Then 
        $f_{S,n,k}$ from Equation~\eqref{eq:f_{S,n,k}}
        is squarefree of degree $k+|S|$ and can be written as $$f_{S,n,k}=\sum_{\substack{u\; \mathrm{ squarefree}\\ 
            \deg(u)=k+|S| \\
            x_S \mid u}} u.$$
        \item 
        Let $k\leq d\leq n$. Then $\mathrm{SFP}((I_{n,k})_{(d)})$ is generated as a vector space by the polynomials $f_{S,n,k}$ with $|S|=d-k$.
    \end{enumerate}
\end{lemma}

\begin{proof}
(1) First observe that $\mathrm{SFP}((x_1+\cdots + x_n)^k)$ is $k!$ times the sum of all squarefree terms of degree $k$. Multiplying by $x_S$, applying $\mathrm{SFP}$ again, and dividing by $k!$, we obtain the desired expression.

    (2) By Definition~\ref{def:g_S_n_k} and item (1), we have $f_{S,n,k}\in \mathrm{SFP}((I_{n,k})_{(d)})$ if $|S|=d-k$. In order to show that these polynomials generate $\mathrm{SFP}((I_{n,k})_{(d)})$, first note that $\{(1/k!)x_S\mid x_S \;\mathrm{squarefree}\;\mathrm{and}\deg(x_S)=d-k\}$ generates $\mathrm{SFP}(\mathbf{k}[x_1,\ldots,x_n]_{(d-k)})$. So by Definition~\ref{def:g_S_n_k}, the vector space $V:=\mathrm{SFP}(\mathbf{k}[x_1,\ldots,x_n]_{(d-k)}\cdot (x_1+\cdots +x_n)^k)$ is generated by the $f_{S,n,k}$ with $|S|=d-k$. Now, since $\mathrm{SFP}((I_{n,k})_{(d)})=\mathrm{SFP}(((x_1+\cdots +x_n)^k)_{(d)})=V$, we are done.
\end{proof}

\begin{example}\label{ex:GenSysHomogPart}
    Consider the linear form $\ell=x_1+\cdots + x_5$ in $\mathbf{k}[x_1,x_2,x_3,x_4,x_5]$. We then have $I_{5,2}=(x_1^2,x_2^2,x_3^2,x_4^2,x_5^2,\ell^2)$. We obtain the generating system $\{f_{\{1\},n,k},\ldots,f_{\{5\},n,k}\}$ of $(I_{5,2})_{(3)}$ as a $\mathbf{k}$-vector space, where
    \begin{align*}
        f_{\{1\},n,k}&=\frac{1}{2}\mathrm{SFP}(x_1 \ell^2) = x_1 x_2 x_3 + x_1 x_2 x_4 + x_1 x_2 x_5 + x_1 x_3 x_4 + x_1 x_3 x_5 + x_1 x_4 x_5 \\&=x_1 e_2(x_2,x_3,x_4,x_5),\\
        f_{\{2\},n,k}&=\frac{1}{2}\mathrm{SFP}(x_2 \ell^2)=x_1 x_2 x_3 + x_1 x_2 x_4 + x_2 x_3 x_4 + x_1 x_2 x_5 + x_2 x_3 x_5 + x_2 x_4 x_5\\&=x_2 e_2(x_1,x_3,x_4,x_5),\\
        f_{\{3\},n,k}&=\frac{1}{2}\mathrm{SFP}(x_3 \ell^2)=x_1 x_2 x_3 + x_1 x_3 x_4 + x_2 x_3 x_4 + x_1 x_3 x_5 + x_2 x_3 x_5 + x_3 x_4 x_5\\&=x_3 e_2(x_1,x_2,x_4,x_5),\\
        f_{\{4\},n,k}&=\frac{1}{2}\mathrm{SFP}(x_4 \ell^2)=x_1 x_2 x_4 + x_1 x_3 x_4 + x_2 x_3 x_4 + x_1 x_4 x_5 + x_2 x_4 x_5 + x_3 x_4 x_5\\&=x_4 e_2(x_1,x_2,x_3,x_5),\\
        f_{\{5\},n,k}&=\frac{1}{2}\mathrm{SFP}(x_5 \ell^2)=x_1 x_2 x_5 + x_1 x_3 x_5 + x_2 x_3 x_5 + x_1 x_4 x_5 + x_2 x_4 x_5 + x_3 x_4 x_5 \\&=x_5 e_2(x_1,x_2,x_3,x_4).\\
    \end{align*}
\end{example}

We will show that the polynomials $g_{A,n,k}$ from Equation~\eqref{eq:g_{A,n,k}-intro} belong to the ideal $I_{n,k}$ by writing them as linear combinations of the basis polynomials $f_{S,n,k}$.
\begin{lemma}\label{lem:Combinations_g_S_n_k}
    Let $k\leq d\leq n$. For any $S\subseteq [n]$ with $|S|=d-k$ let $\lambda_S\in \mathbf{k}$ be a scalar. Then
    $$\sum_{|S|=d-k}\lambda_S f_{S,n,k}=\sum_{\substack{u\;\mathrm{squarefree}\\ \deg(u)=d}}\left(\sum_{\substack{x_S\mid u\\ |S|=d-k}} \lambda_S\right)u.$$
\end{lemma}

\begin{proof}
    Write $f=\sum_{|S|=d-k}\lambda_S f_{S,n,k}$. As a $\mathbf{k}$-linear combination of squarefree polynomials of degree $d$, $f$ is also squarefree of degree $d$. By Lemma~\ref{lem:GenSysDegreewiseSFP}, the coefficient of a squarefree term $u$ of degree $d$ in the polynomial $f_{S,n,k}$ is $\lambda_S$ if $x_S\mid u$ and zero otherwise. Thus the coefficient of $u$ in $f$ is the sum of all $\lambda_S$ where $x_S\mid u$, as claimed.
\end{proof}

\begin{definition}\label{def:f_A_n_k}
    Let $k\leq d\leq n$, and $A=\{i_1,\ldots,i_d\}\subseteq [n]$ with $\max(A)\leq 2d-k$. 
    We define
\begin{equation}\label{eq:g}
g_{A,n,k}=\sum_{\substack{u\;\mathrm{squarefree}\\ \deg(u)=d\\ \mathrm{supp}(u)\cap (\{1,\ldots,2d-k\}\setminus A))=\emptyset}} u.
\end{equation}
\end{definition}

Note that this is the same definition for $g_{A,n,k}$ in terms of elementary symmetric polynomials as given in Equation~\eqref{eq:g_{A,n,k}-intro} 
since for any set $S\subseteq [n]$, we have that
\[
S\cap (\{1,\dots, 2d-k\}\setminus A) = \emptyset \iff S\cap \{1,\dots, 2d-k\} \subseteq A.
\]

\begin{example}\label{ex:f_A_n_k}
    For the values $n=5$, $k=2$, and $d=3$, consider the set $A=\{1,3,4\}$. Note that $\max(A)=4\leq 4=2\cdot d-k$. In this case, $g_{A,n,k}$ is given by
    $$e_3(x_1,x_3,x_4,x_5) = x_1 x_3 x_4 + x_1 x_3 x_5 + x_1 x_4 x_5 + x_3 x_4 x_5,$$
     the elementary symmetric polynomial $e_3(x_1,x_3,x_4,x_5)$ of degree $d$ supported on the variables indexed by $A\cup \{5\}=A\cup \big([n]\setminus \{1,\ldots,2d-k\} \big)$.
\end{example}

With this preparation, we are now ready to establish that $g_{A,n,k}\in I_{n,k}$, which is a major step needed for the proof of Theorem~\ref{thm:GBasis}.

\begin{proposition}\label{prop:ExplicitCoefficients_f_A_n_k} 
    Let $k\leq d\leq n$, and let $A\subseteq [n]$ with $\max(A)\leq 2d-k$. Write $\mathcal{T}_i=\{ S\subseteq [n]: |S|=d-k, \text{ and } |S\cap (\{1,\dots, 2d-k\}\setminus A)|=i \}$. Then we have
    $$g_{A,n,k}=\sum_{i=0}^{d-k} \lambda_i \left(\sum_{S\in\mathcal{T}_i} f_{S,n,k}\right)$$
    for some coefficients $\lambda_0,\ldots,\lambda_{d-k}\in \mathbf{k}$. In particular, $g_{A,n,k}\in I_{n,k}$.
\end{proposition}
\begin{proof}
Let $C=\{1,\ldots,2d-k\}\setminus A$. Write 
    $$f=\sum_{i=0}^{d-k} \lambda_i \left(\sum_{S\in\mathcal{T}_i} f_{S,n,k}\right).$$
    By construction, $f$ is a sum of squarefree monomials of degree $d$. Let $u$ be a squarefree term with $\deg(u)=d$. By Lemma~\ref{lem:Combinations_g_S_n_k}, its coefficient in $f$ is $\sum_{i=0}^{d-k}  |\{S\in\mathcal{T}_i : x_S\mid u \}|\cdot \lambda_i$. On the other hand, the coefficient of $u$ in $g_{A,n,k}$ is $1$  if $|\mathrm{supp}(u)\cap C|=0$, and zero otherwise.
    Thus, $f=g_{A,n,k}$ holds if and only if the coefficients $\lambda_i$ solve the inhomogeneous linear system of $\binom{n}{d}$ equations (labeled by the terms $u$)
    \begin{equation}\label{eq:InhomogLinSys_1}
        \sum_{i=0}^{d-k}  |\{S\in\mathcal{T}_i : x_S\mid u \}|\cdot \lambda_i=\begin{cases}
        1, & \mathrm{if}\; |\mathrm{supp}(u)\cap C|=0 \\
        0, & \mathrm{if}\; |\mathrm{supp}(u)\cap C|>0.
        \end{cases}
    \end{equation}
    Note that if $|\mathrm{supp}(u)\cap C|=j$ and $S\in \mathcal{T}_i$ with $i>j$, we have $x_S\nmid u$ by the definitions of $\mathcal{T}_i$ and $C$ because $|\mathrm{supp}(x_S)\cap C|=i$. Thus, we can rewrite the system~\eqref{eq:InhomogLinSys_1} as 
    \begin{equation}\label{eq:InhomogLinSys_2}
        \sum_{i=0}^{j}  |\{S\in\mathcal{T}_i : x_S\mid u \}|\cdot \lambda_i=\begin{cases}
        1, & \mathrm{if}\; |\mathrm{supp}(u)\cap C|=0 \\
        0, & \mathrm{if}\; |\mathrm{supp}(u)\cap C|=j>0.
        \end{cases}
    \end{equation}
    We now claim that the system~\eqref{eq:InhomogLinSys_2} contains only $d-k+1$ distinct equations. First, it is easy to see that for $|\mathrm{supp}(u)\cap C|=0$ we have $|\{S\in\mathcal{T}_0 : x_S\mid u \}|=\binom{d}{d-k}$ independently of $u$. Now, consider $u$ with $|\mathrm{supp}(u)\cap C|=j>0$. Note that $|\mathrm{supp}(u)\setminus C|=d-j$. Furthermore, consider $S\in\mathcal{T}_i$ with $0\leq i\leq j$. Then $|S\cap C|=i$ and $|S\setminus C|=d-k-i$. The condition $x_S\mid u$ is then equivalent to satisfying both $S\cap C\subseteq \mathrm{supp}(u)\cap C$ and $S\setminus C\subseteq \mathrm{supp}(u)\setminus C$. The number of sets $S\in\mathcal{T}_i$ that fulfill this is $\binom{j}{i}\binom{d-j}{d-k-i}$ independently of $u$.

    The system~\eqref{eq:InhomogLinSys_2} now simplifies to 
    \begin{equation}\label{eq:InhomogLinSys_3}
        \left\{\begin{array}{rl}
             \binom{d}{d-k} \lambda_0  & =1 \\
            \sum_{i=0}^1\binom{1}{i}\binom{d-1}{d-k-i} \lambda_i & =0 \\
             \vdots & \\
             \sum_{i=0}^{d-k}\binom{d-k}{i}\binom{k}{d-k-i} \lambda_i & =0
        \end{array}\right. .
    \end{equation}
    This is a lower triangular system with diagonal elements $\binom{d-j}{d-k-j}=\binom{d-j}{k}\neq 0$ (recall that $k\leq d$ and $d-j\geq d-(d-k)=k$). Hence the system has a unique solution and $g_{A,n,k}$ can be written as desired.
\end{proof}

In the setting of Proposition~\ref{prop:ExplicitCoefficients_f_A_n_k}, we have the following corollary.
\begin{corollary}\label{cor:SpecialCase_n=2d-k}
   If $n=2d-k$, then the set $\{f_{S,n,k}: S\subseteq[n],|S|=d-k\}$ is a vector space basis of $\mathrm{SFP}(K[x_1,\ldots,x_n]_{(d)})=\mathrm{SFP}((I_{n,k})_{(d)})$. 
\end{corollary}
\begin{proof}
    Recall the notation $A=\{i_1,\ldots,i_d\}$. Since $g_{A,n,k}=x_{i_1}\cdots x_{i_d}\in I_{n,k}$, the degree $d$ part $\mathrm{SFP}((I_{n,k})_{(d)})$ contains a vector space basis of $V=\mathrm{SFP}(K[x_1,\ldots,x_n]_{(d)})$. The dimension of $V$ is $\binom{n}{d}$. Moreover, by Lemma~\ref{lem:GenSysDegreewiseSFP}, the set $\mathcal{B}:=\{f_{S,n,k}: S\subseteq[n],|S|=d-k\}$ generates $V$. The statement now follows from $|\mathcal{B}|=\binom{n}{d-k}=\binom{n}{n-d+k}=\binom{n}{d}=\dim(V)$.
\end{proof}

\begin{lemma}\label{lem:BinomialCoeffsAlternatingSums}
    The unique solution of the inhomogeneous system of linear equations~\eqref{eq:InhomogLinSys_3} is $\lambda_i=(-1)^{i}\frac{ k }{ (k+i)\binom{d}{k+i} }$ for $0\leq i \leq d-k$.
\end{lemma}
\begin{proof}
For $i=0$, the proposed $\lambda_0=\binom{d}{k}^{-1}=\binom{d}{d-k}^{-1}$ satisfies the first equation in \eqref{eq:InhomogLinSys_3} as desired. Pick now any $j$ with $1\leq j \leq d-k$. Inserting our suggested formula for $\lambda_i$ into the equations in \eqref{eq:InhomogLinSys_3}, we must show that
\begin{equation}
\label{eq:explicit_coeff_1}
\sum_{i=0}^{j}(-1)^i\binom{j}{i}\binom{d-j}{d-k-i}\frac{k}{(k+i)\binom{d}{k+i}} = 0.
\end{equation}
Rewriting the left hand side of \eqref{eq:explicit_coeff_1} gives the equivalent equation
\begin{equation}
\label{eq:explicit_coeff_2}
\frac{k}{j\binom{d}{j}}\sum_{i=0}^{j}(-1)^i\binom{j}{i}\binom{k+i-1}{j-1} = 0.
\end{equation}
To prove that this is true, we will use the notion of finite differences. Let $\Delta f(x)=f(x+1)-f(x)$ be the first finite difference of $f$ and 
\[
\Delta^n(f(x))=\Delta(\Delta^{n-1}f(x)) = \sum_{i=0}^n(-1)^{n-i}\binom{n}{i}f(x+i)
\]
be the $n$-th finite difference of $f$. Using this notation, we see that equation \eqref{eq:explicit_coeff_2} can be expressed as
\[
\frac{k(-1)^j}{j\binom{d}{j}}\Delta^j(f(k))=0
\]
for $f(k)=\binom{k-1}{j-1}$. But as $f(k)$ is a polynomial in $k$, of degree $j-1$, and $\Delta$ takes a polynomial of degree $j$ to one of degree at most $j-1$, we get that $\Delta^j(f(k))=0$ and the equality is proven.
\end{proof}

We are now ready to give an explicit description of the elements $g_{A,n,k}$ in terms of the generators of $I_{n,k}$.

\begin{theorem}
\label{thm:f_description}
Let $f_{S,n,k}\in I_{n,k}$ for $S\subseteq [n]$ be the squarefree part of the polynomial $x_S(x_1+\cdots+x_n)^k$. Then the elements $g_{A,n,k}$ can be written as
$$g_{A,n,k}=\sum_{i=0}^{d-k} (-1)^{i}\frac{ k }{ (k+i)\binom{d}{k+i} } \sum_{S\in\mathcal{T}_i(A)} f_{S,n,k}$$
where $d=|A|$ and
$$\mathcal{T}_i(A)=\{S\subseteq [n]: |S|=d-k \text{ and }|S\cap\{1,\ldots,2d-k\}\setminus A| =i \}.$$
\end{theorem}

\begin{proof}
The proof follows directly by combining Lemma~\ref{lem:BinomialCoeffsAlternatingSums} with Proposition~\ref{prop:ExplicitCoefficients_f_A_n_k}.
\end{proof}
 We illustrate Theorem~\ref{thm:f_description} with an example.
 \begin{example}\label{ex:InhomogSystem_n=5_d=3_k=2}
    Let $n=5$, $k=2$, and $d=3$. Moreover, let $A=\{1,3,4\}$. Then $g_{A,n,k}=x_1 x_3 x_4 + x_1 x_3 x_5 + x_1 x_4 x_5 + x_3 x_4 x_5$. Since $d-k=1$, we look at the set $\{ f_{\{1\},n,k},\ldots, f_{\{5\},n,k}\}\subset (I_{n,k})_{(3)}$, as described in Example~\ref{ex:GenSysHomogPart}. 
     Since $2d-k=4$, we get $C=\{1,\ldots,4\}\setminus A=\{2\}$. Moreover, $\mathcal{T}_0=\{ \{1\},\{3\},\{4\},\{5\} \}$ and $\mathcal{T}_1=\{ \{2\} \}$. We obtain an inhomogeneous linear system of equations
    \begin{equation}\label{eq:InhomogLinSys_example}
        \left\{\begin{array}{ll}
             \binom{3}{1} \lambda_0  & =1 \\[1em]
             \binom{2}{1} \lambda_0 + \binom{2}{0} \lambda_1 & =0
        \end{array}\right. ,
    \end{equation}
    whose solution is clearly $\lambda_0=\frac{1}{3}$, $\lambda_1=-\frac{2}{3}$. On the other hand, one can easily verify that $\lambda_0=(-1)^{0}\frac{ 2 }{ 2\binom{3}{2} }$ and $\lambda_1=(-1)^{1}\frac{ 2 }{ 3\binom{3}{3} }$ as stated in Lemma~\ref{lem:BinomialCoeffsAlternatingSums}, and that 
    $$g_{A,n,k}=\frac{1}{3}\left(f_{\{1\},n,k}+f_{\{3\},n,k}+f_{\{4\},n,k}+f_{\{5\},n,k}\right)-\frac{2}{3}f_{\{2\},n,k},$$
    as stated in Theorem~\ref{thm:f_description}.
\end{example}

\subsection{Hilbert series via Lattice paths}

In what follows, we will prove that the reduced degree reverse lexicographic Gr\"{o}bner basis of $I_{n,k}$ is the union of $\{x_1^2,\ldots,x_n^2\}$ with the set of $g_{A,n,k}$ whose leading term is minimal with respect to division. We will need a counting argument, and, to this end, we introduce a certain type of lattice paths.

\begin{definition}\label{def:LatticePaths}
 An \emph{$(N,E)$-lattice path} is a path on the lattice $\mathbb{Z}^2$ that begins at $(0,0)$ and consists only of northward steps (in the direction $(0,1)$, denoted $N$) and eastward steps (in the direction $(1,0)$, denoted $E$).

There exists a bijection $\tau$ that maps an $(N,E)$-lattice path of length $n$ to the squarefree monomial $\prod_{j \in J} x_j$, where the subset $J \subseteq [n]$ contains an index $j$ if and only if the $j$-th step in the path is north.
\end{definition}

The lattice path associated to $\ini(g_{\{1,3,4\},5,2})=x_1 x_3 x_4$ is illustrated in Figure~\ref{fig:LatticePath_134_n=5_k=2}.
\begin{figure}[h!]
    \centering
    \begin{tikzpicture}
        \draw[step=1cm,gray!30,thin] (0,0) grid (5,5);
        \draw[thick, red, opacity=0.7] (0,2) -- (2.5,4.5) node[above] {$y = x + 2$};

        \draw[thick, black] (0,0) -- (0,1) -- (1,1) -- (1,3) -- (2,3);

        \node[below left] at (0,0) {$(0,0)$};
        \node[above right] at (1,1) {Path $\tau^{-1}(\text{in}(g_{\{1,3,4\},5,2}))$};
    \end{tikzpicture}
    
    \caption{The $(N,E)$-lattice path associated to $x_1 x_3 x_4$.}
    \label{fig:LatticePath_134_n=5_k=2}
\end{figure}

\begin{lemma}\label{lem:LatticePathBijections}
	Let $k\leq d\leq n$. Consider polynomials $g_{A,n,k}$ for the sets $A=\{i_1,\ldots,i_d\}\subseteq[n]$ with $\max(A)\leq 2d-k$ as in \eqref{eq:g}.  
 Then, we have that 
    \begin{enumerate}
	   \item $\ini(g_{A,n,k})=x_{i_1}\cdots x_{i_d}$,
	   \item the $(N,E)$-lattice path $\tau^{-1}(\ini(g_{A,n,k}))$ touches the line $y=x+k$,
	   \item if $P$ is an $(N,E)$-lattice path of length $n$ that intersects the line $y=x+k$, then there exists a polynomial $g_{A,n,k} \in I_{n,k}$ such that $\ini(g_{A,n,k})$ divides $\tau(P)$.
    \end{enumerate}
\end{lemma}
\begin{proof}
(1) By the definition of $g_{A, n, k}$, the term $x_{i_1} \cdots x_{i_d}$ appears as one of the summands in $g_{A, n, k}$, and all other terms in its support must include at least one variable with an index greater than $\max(A)$.
       
(2) Since $\max(A) \leq 2d-k$, there are $d$ north steps and $d-k$ east steps among the first $2d-k$ steps in $\tau^{-1}(\ini(g_{A,n,k}))$. This difference of $k$ additional north steps forces the path to intersect the line $y=x+k$ after $2d-k$ steps.

(3) In a path $P$ that intersects the line $y=x+k$, there exists a minimal starting segment (of length at least $k$) that touches the line. Within this segment, there must be exactly $k$ more steps north than east. Let $\ell$ denote the length of this segment; then there exists a positive integer $d$ such that $\ell = 2d - k$. Consider the path $\tilde{P}$, which consists of the first $\ell$ steps of $P$ followed by $n - \ell$ steps east. By construction, there is a polynomial $g_{A,n,k} \in I_{n,k}$ such that $\ini(g_{A,n,k}) = \tau(\tilde{P})$. Since $\tau(\tilde{P})$ divides $\tau(P)$, the claim follows.
\end{proof}

We now proceed to examine the number of elements $g_{A,n,k}$ such that $A$ is minimal with respect to inclusion. Their leading terms minimally generate $\mathrm{SFP}(\ini(I_{n,k}))$ because strict inclusions $A_1\subset A_2$ translate to strict divisibilites $x_{A_1}=\ini{(g_{A_1,n,k})}\mid \ini{(g_{A_2,n,k})}=x_{A_2}$ and vice versa. We start with the following definitions.

\begin{definition}
Let $(a_i)_{i=0}^{\infty}$ be a sequence of integers. The $k$-fold \emph{self-convolution} of $(a_i)_{i=0}^{\infty}$, denoted $(a_i^k)_{i=0}^{\infty}$, is defined as the sequence of coefficients of the power series 
$\textstyle{\left(\sum_{i=0}^{\infty}a_it^i \right)^{k+1}
}$.
\end{definition}

One sequence that is central to our discussion is the Catalan numbers.
\begin{definition}
The \emph{Catalan numbers} $(C_n)_{n=0}^{\infty}$ are given by
$C_n = \binom{2n}{n} - \binom{2n}{n-1}$.
The $r$-th number of the Catalan $(k-1)$-fold \emph{convolution} is denoted by $C_r^{k-1}$. 
\end{definition}

\begin{corollary}\label{cor:ConvolutionCatalanNumbers}
The number of polynomials $g_{A,n,k}$ such that $A$ is minimal with respect to inclusion, and of degree $k+r$, is zero if $n < 2r + k$; otherwise, it is given by $C_r^{k-1}$. 
In particular, for $k=1$ and $k=2$, it is the $r$-th and $(r+1)$-th Catalan number, respectively.
\end{corollary}

\begin{proof}
By Lemma~\ref{lem:LatticePathBijections}, the 
polynomials $g_{A,n,k}$ of degree $k+r$ are in bijection with $(N,E)$-lattice paths that take at most $n$ steps and touch the line $y=x+k$ exactly once, specifically at the last step after $k+r$ steps north. Since such a path comprises a total of $2r+k$ steps, there can be no such paths if $n < 2r+k$, resulting in no Gröbner basis elements of degree $r+k$. For $n \geq 2r+k$, we find that, disregarding the last step, these paths are also in bijection with paths that remain below the line $y=x+k-1$ and terminate on that same line after $k-1+r$ steps north.

Next, by \cite[Corollary 16]{Tedford}, $C_i^j$ represents the number of $(N,E)$-lattice paths that start at $(i,0)$, do not cross the line $y=x$, and end at $(i+j,i+j)$. By shifting all such paths left by $i$ steps, we can also interpret $C_i^j$ as counting $(N,E)$-paths starting at the origin that remain below the line $y=x+i$ and terminate at $(j,i+j)$. By combining this count with the previously established enumeration for the number of polynomials $g_{A,n,k}$ of degree $k+r$, we arrive at the complete enumeration. Specifically, for $k=1$ we get $C_r^0 = C_r$ polynomials in degree $r$ and for $k=2$, the claim follows from the well-known fact that $C_n^1=C_{n+1}$; 
see e.g.~\cite[Equation (1.1)]{Catalan_Stanley}.  
\end{proof}

The following lemma presents our main counting argument and serves as the final step needed to complete the main part of the proof of Theorem~\ref{thm:GBasis}.

\begin{lemma}\label{lem:CountingArgument}
Let $d$ be an integer with $0 \leq d \leq n$, and consider the terms of degree $d$ outside the degree reverse lexicographic initial ideal $\ini(I_{n,k})$. If $2d-k\geq n$, there are no such terms. Otherwise, these terms are squarefree, and their number is at most $\binom{n}{d} - \binom{n}{d - k}$.
\end{lemma}

\begin{proof}
First, we observe that all such terms must be squarefree, since $x_i^2 \in I_{n,k}$ for all $i = 1, \dots, n$. We will show that the number of squarefree monomials of degree $d$ that are not divisible by any $\ini(g_{A,n,k})$ when $2d-k<n$ is precisely $\binom{n}{d} - \binom{n}{d - k}$. Since $g_{A,n,k} \in I_{n,k}$ by Proposition~\ref{prop:ExplicitCoefficients_f_A_n_k}, this establishes the desired bound on the number of terms outside $\ini(I_{n,k})$ in that case.

By Lemma~\ref{lem:LatticePathBijections}, the number of squarefree monomials of degree $d$ not divisible by any $\ini(g_{A,n,k})$ corresponds to the number of $(N, E)$-lattice paths of length $n$ taking $d$ steps north, remaining strictly below the line $y = x + k$. If $2d-k\geq n$, then all path of length $n$ taking $d$ steps north do touch the line $y=x+k$. Hence there are no monomials outside $\ini(I_{n,k})$ in that case. Assume $2d-k<n$ from now on. Given that there are $\binom{n}{d}$ such $(N, E)$-lattice paths of length $n$ with $d$ north steps, it remains to prove that the number of paths touching the line $y = x + k$ in this case is exactly $\binom{n}{d - k}$. We will establish this by induction on $n$.

The case for $n = 1$ is straightforward, so assume $n > 1$ and that the formula holds for $n - 1$. Further, suppose the point $(n - d, d)$ lies below the line $y = x + k$, ensuring that not all paths of length $n$ with $d$ north steps are forced to touch this line.
The key observation is that $(N, E)$-paths ending at $(n - d, d)$ correspond bijectively to paths ending one step earlier at either $(n - d - 1, d)$ or $(n - d, d - 1)$. This correspondence remains valid even if we enforce the condition of touching the line $y = x + k$. By the induction hypothesis, the number of such paths is given by
\[
\binom{n - 1}{d - k} + \binom{n - 1}{d - 1 - k} = \binom{n}{d - k},
\]
as required.
The one case needing extra care is when $(n - d - 1, d)$ lies exactly on the line $y = x + k$, since we cannot directly apply the induction hypothesis in this situation. However, in this case, all $\binom{n - 1}{d}$ paths necessarily touch the line. Since $\binom{n - 1}{d} = \binom{n - 1}{d - k}$ when $(n - d - 1, d)$ lies on $y = x + k$, the formula still holds.
\end{proof}

\begin{corollary}\label{cor:Hilbert_series}
The Hilbert series of $R/I_{n,k}$ is given by
\[
\left[(1+t)^n(1-t^k)\right]
\]
where the brackets indicate truncation at the first non-positive coefficient.
\end{corollary}

\begin{proof}
Since 
\[
\left[(1+t)^n(1-t^k)\right] = \left[\frac{(1-t^2)^n}{(1-t)^n}(1-t^k)\right]\!,
\]
Fr\"oberg's result \cite{Froberg} implies that this is the smallest the Hilbert series of $R/I_{n,k}$ can be. Next, the coefficient of $t^d$ of this polynomial is given by 
\[
\max \left\{\binom{n}{d}-\binom{n}{d-k},0\right\} = \begin{cases}
\binom{n}{d}-\binom{n}{d-k} & \text{ if } 2d-k< n \\
0 & \text{ otherwise}.
\end{cases}
\]
However, by Lemma \ref{lem:CountingArgument}, the Hilbert series of $R/I_{n,k}$ cannot be larger than this. Therefore, equality holds.
\end{proof}

\subsection{Proof of the main theorem}
We are now ready to state our main result. 
\begin{theorem}\label{thm:GBasis} 
Consider the family $\mathcal{A}$ of inclusion-minimal members of the family $\mathcal{A}'$ of subsets $A \subseteq \{1, \ldots, n\}$ satisfying $\max(A) = 2|A| - k$ for some $k\geq 2$.
Then, the reduced Gröbner basis of the ideal $I_{n,k} = (x_1^2, \dots, x_n^2, (x_1 + \cdots + x_n)^k)$ with respect to any term ordering that, after any permutation of the first $k$ variables and a suitable set of transpositions of the form $(k+2i - 1,k+ 2i)$ for $i\geq 1$, satisfies $x_1 \succ \cdots \succ x_n$, is given by
$$G_{n,k}=\{x_1^{2},\ldots,x_n^2 \}\cup \bigcup_{d=k}^{k+\lfloor (n-k)/2\rfloor}\left\{ g_{A,n,k}\mid A\in\mathcal{A}, |A|=d \right\},$$
where for $|A|=d$,
\begin{equation*}\label{eq:g_{A,n,k}}
    g_{A,n,k} = e_d(x_{i_1},\dots, x_{i_{n-d+k}})
\end{equation*}
is the elementary symmetric polynomial of degree $d$ in the variables indexed by the set $\{i_1,\dots, i_{n-d+k}\}=A\cup \{2d-k+1,\dots, n\}$ with leading term $x_A=\prod_{a\in A}x_a$. In particular, $g_{A,n,k}$ is supported on $\binom{n-d+k}{d}$ terms.

Moreover, for fixed $k$ the sequence of cardinalities $|\{A\in\mathcal{A}: |A|=d\}|$ is a $(k-1)$-fold convolution of the sequence of Catalan numbers.
\end{theorem}
\begin{proof} 
We start by using any term ordering with $x_1 \succ \cdots \succ x_n$. First, by Lemma~\ref{lem:LatticePathBijections} and Corollary~\ref{cor:Hilbert_series}, the squarefree leading terms of polynomials in $I_{n,k}$ are in bijection with $n$-step $(N,E)$-lattice paths that touch the line $y = x + k$, and each of these terms is divisible by some leading term $\ini(g_{A,n,k})$. Among these leading terms, the minimal ones with respect to division are those that touch the line $y = x + k$ exactly once, occurring after their last step north. Moreover, the penultimate step cannot be east; if it were, there would be an earlier intersection with the line. This establishes the defining conditions of the set $\mathcal{A}$. Regarding non-squarefree leading terms, it is clear that they are divisible by some $x_i^2$. By Proposition~\ref{prop:ExplicitCoefficients_f_A_n_k}, $G_{n,k}\subset I_{n,k}$. Thus, $G_{n,k}$ is a minimal Gr\"{o}bner basis of $I_{n,k}$.

\medskip
To show that $G_{n,k}$ is reduced, first observe that the leading coefficients of its elements are $1$. Now, consider \(g_{A,n,k} \in G_{n,k}\). Let \((a,b) \in \mathbb{Z}^2\) be the lattice point where \(\tau^{-1}(\ini(g_{A,n,k}))\) touches the line \(y = x + k\). Any term \(u \in g_{A,n,k} - \ini(g_{A,n,k})\) can be expressed as 
$u=x_{j_1}\cdots x_{j_m}\ini(g_{A,n,k})/(x_{\ell_1}\cdots x_{\ell_m})$,
where \(0 < m \leq |A|\), \(j_1, \ldots, j_m > \max(A)\), and \(\ell_1, \ldots, \ell_m \in A\). Consequently, the lattice path \(\tau^{-1}(u)\) reaches the point \((a + m, b - m)\) after \(\max(A)\) steps. The closest it could approach the line \(y = x + k\) is at the point \((a + m, b)\) after \(\max(A) + m\) steps. However, this point lies below the line since \(m > 0\) and \((a, b)\) is on the line. Therefore, we conclude that \(u \notin \ini(I_{n,k})\).

The claim on the sequence of cardinalities $|\{A\in\mathcal{A}: |A|=d\}|$ for fixed $k$ follows from Corollary~\ref{cor:ConvolutionCatalanNumbers}.
Moreover, Lemma~\ref{lem:Invariance_G_n_k} below states the conditions a term ordering $\prec'$ must satisfy for $\ini_{\prec'}(I_{n,k})=\ini(G_{n,k})$ to hold. It shows that only the blockwise ranking $\{x_1,\ldots,x_k\}\succ \{x_{k+1},x_{k+2}\}\succ \cdots$ matters, and this establishes the statement on permutations. 
\end{proof}

\begin{remark}\label{rem:k=1}
The reason for omitting the case $k = 1$ from Theorem \ref{thm:GBasis} is due to notational challenges. But it is straightforward to give the Gr\"obner basis also in this case. To this end, we introduce an extra variable $x_0$ with the property that $x_0 \succ x_i$ for $i = 1, \ldots, n$ and consider the ideal  $$I=(x_0^2,x_1^2, \ldots,x_n^2,(x_0+x_1+\cdots+x_n)).$$ After reducing $x_0^2$ with respect to the linear element we get $$I=((x_1+\cdots+x_n)^2, x_1^2,\ldots,x_n^2,(x_0+x_1+\cdots+x_n)).$$ Since $x_0$ only occurs in the last linear polynomial, and is leading, it follows that the reduced Gr\"obner basis for $I$ is equal to $$\{e_1(x_0,\ldots,x_n)\} \cup G_{n,2}.$$
\end{remark}

\begin{remark}
Worth noticing is that although the elements of the Gr\"obner basis only have $0,1$-coefficients, our result only holds in characteristic $0$ in general. This is explicit from the description of the Gr\"obner basis elements given in Theorem~\ref{thm:f_description},  and implicit from results showing that the Lefschetz properties for monomial complete intersections is dependent on the characteristic of the underlying field~\cite{Kustin, LUNDQVIST2019213}.
\end{remark}
\begin{example}\label{ex:GB_n=5_k=2}
    As in Example~\ref{ex:InhomogSystem_n=5_d=3_k=2}, we set $n=5$ and $k=2$. The reduced Gr\"{o}bner basis of $I_{5,2}$ with respect to the degree reverse lexicographic ordering is 
    $$G_{5,2}=\{x_1^2,x_2^2,x_3^2,x_4^2,x_5^2,g_{\{1,2\},5,2},g_{\{1,3,4\},5,2},g_{\{2,3,4\},5,2}  \},$$
    where
    \begin{align*}
        g_{\{1,2\},5,2}&= e_2(x_1,x_2,x_3,x_4,x_5) = x_1 x_2 + x_1 x_3 + \cdots + x_4 x_5,\\ 
        g_{\{1,3,4\},5,2}&= e_3(x_1,x_3,x_4,x_5) = x_1 x_3 x_4 + x_1 x_3 x_5 + x_1 x_4 x_5 + x_3 x_4 x_5,\\
        g_{\{2,3,4\},5,2}&= e_3(x_2,x_3,x_4,x_5) = x_2 x_3 x_4 + x_2 x_3 x_5 + x_2 x_4 x_5 + x_3 x_4 x_5.
    \end{align*}
\end{example}

\subsection{The Gr\"obner basis for a general term ordering}

In the remainder of this section, we investigate all the Gr\"{o}bner bases of $I_{n,k}$, showing that they have the same general structure as $G_{n,k}$, but depend on the way the variables are ranked.  

\begin{remark}\label{rem:G_n_k_NotUniversalInGeneral}
Note that $G_{5,2}$, and indeed any Gr\"{o}bner basis of the form $G_{n,k}$ is also a reduced Gr\"{o}bner basis with respect to the lexicographic ordering with $x_1\succ x_2\succ\cdots$. This is because the leading terms of the elements $g_{A,n,k}$ are the same for the degree reverse lexicographic and lexicographic orderings.

However, $G_{n,k}$ is in general not a universal Gr\"{o}bner basis: For instance, in $G_{5,2}$ (see Example~\ref{ex:GB_n=5_k=2}), the polynomials $g_{\{1,3,4\},5,2}$ and $g_{\{2,3,4\},5,2}$ share the same leading term, $x_3 x_4 x_5$, with respect to the degree reverse lexicographic ordering with $x_5\succ x_4\succ x_3\succ x_2 \succ x_1$.
\end{remark}

We aim to compute the number of different reduced Gr\"{o}bner bases of the ideal $I_{n,k}$. We write $N_{n,k}$ for this number. It turns out that certain features of the variable ranking induced by a given term order $\prec$ suffice to determine the Gr\"{o}bner basis of $I_{n,k}$ with respect to $\prec$.

\begin{definition}\label{def:BlockVariableRanking}
 Let $\prec$ be a term ordering on $\mathbf{k}[x_1,\ldots,x_n]$ and let $X_1,\ldots,X_m$ be subsets of the variables $X=\{x_1,\ldots,x_n\}$ that partition $X$. We say that $\prec$ induces the \emph{block variable ranking} $X_1\succ X_2\succ \cdots\succ X_m$ if, for any integers $1 \leq i < j \leq m$, with $x \in X_i$ and $y \in X_j$, it holds that $x \succ y$. We call an ordering an \emph{$(\ell_1,\ell_2,\ldots,\ell_m)$-block variable ranking} if $|X_i|=\ell_i$ for $1 \leq i \leq m$. In particular, a $(k,2,\ldots,2,1)$-block variable ranking is any term ordering such that:
\begin{itemize}
    \item $\{x_1,\ldots,x_k\}\succ \{x_{k+1},x_{k+2}\}\succ\cdots \succ \{x_{n-1},x_n\}$ if $n-k \equiv 0 \mod 2$, and
    \item $\{x_1,\ldots,x_k\}\succ \{x_{k+1},x_{k+2}\}\succ\cdots\succ \{x_{n-2},x_{n-1}\}\succ \{x_n\}$ if $n-k \equiv 1 \mod 2$.
\end{itemize}
\end{definition}

\begin{lemma}\label{lem:Invariance_G_n_k}
    Let $\prec$ be a $(k,2,\ldots,2,1)$-block variable ranking. 
    Then the reduced Gr\"{o}bner basis of $I_{n,k}$ with respect to $\prec$ is $G_{n,k}$.
\end{lemma}
\begin{proof}
   We need to show that for each polynomial $g \in G_{n,k}$, its leading term with respect to the term ordering $\prec$ coincides with its leading term with respect to the standard degree reverse lexicographic ordering. This is straightforward for the monomials in $G_{n,k}$.
Now, consider any polynomial $g_{A,n,k} \in G_{n,k}$. Let $A = \{i_1, \ldots, i_d\}$ as usual, where we assume $i_1 < \cdots < i_d$ and $\max(A) = i_d = 2d - k$. Then we have $\max(A) - k = 2d - k - k = 2(d - k) \equiv 0 \mod 2$. This implies that any variable with an index less than or equal to $\max(A)$ is ranked higher by $\prec$ than any variable with an index greater than $\max(A)$. 
Recalling Definition~\ref{def:f_A_n_k}, it follows that $\ini_{\prec}(g_{A,n,k}) = x_{i_1} \cdots x_{i_d}$, as desired.
\end{proof}

\begin{definition}\label{def:G_n_k_sigma}
 Consider a permutation $\sigma \in S_n$ and let $x_{\sigma(1)}, \ldots, x_{\sigma(n)}$ be a permutation of the variables $x_1,\ldots,x_n$. With respect to the degree reverse lexicographic ordering defined by $x_{\sigma(1)} \succ \cdots \succ x_{\sigma(n)}$, we obtain a Gröbner basis for $I_{n,k}$, as outlined in Theorem~\ref{thm:GBasis} denoted by $G_{n,k}^{\sigma}$. The Gröbner basis $G_{n,k}$ from Theorem~\ref{thm:GBasis} can be expressed as $G_{n,k}^{\mathrm{id}}$, for the identity permutation `id'. 
\end{definition}

\begin{proposition}\label{prop:N_n_k_viaMultinomialCoeffs}
The number $N_{n,k}$ of distinct Gröbner bases of $I_{n,k}$ is given by the multinomial coefficient $\binom{n}{k, 2, \ldots, 2, 1}$, where the final entry $1$ is included if and only if $n - k \equiv 1 \mod 2$.
\end{proposition}
\begin{proof}
Since the multinomial coefficient $\binom{n}{k, 2, \ldots, 2, 1}$ counts the number of ordered partitions of an $n$-element set into subsets of cardinalities $k$, $2$, \ldots, $2$, and (possibly) $1$, it can be interpreted as the number of distinct $(k, 2, \ldots, 2, 1)$-block variable rankings of the variables $x_1, \ldots, x_n$. Therefore, it suffices to demonstrate that the reduced Gröbner bases of $I_{n,k}$ with respect to the term orderings $\prec_1$ and $\prec_2$ are equal if and only if $\prec_1$ and $\prec_2$ induce the same $(k, 2, \ldots, 2, 1)$-block variable ranking.

The `if' direction is established in Lemma~\ref{lem:Invariance_G_n_k}.

For the `only if' direction, consider term orderings $\prec_1$ and $\prec_2$ that induce different $(k, 2, \ldots, 2, 1)$-block variable rankings. There exist permutations $\sigma_1, \sigma_2 \in S_n$ such that the Gröbner bases for $\prec_1$ and $\prec_2$ are $G_{n,k}^{\sigma_1}$ and $G_{n,k}^{\sigma_2}$, respectively.
Then, there exists an earliest block from the left where the rankings differ. Let $Y$ be the union of the common blocks to the left in both rankings. If $Y$ is empty, the rankings differ at the first sets, which we denote as $\{w_1,\ldots,w_k\}$ for $\prec_1$ and $\{z_1,\ldots,z_k\}$ for $\prec_2$. Since the degree $k$ squarefree components of $\ini_{\prec_1}(I_{n,k})$ and $\ini_{\prec_2}(I_{n,k})$ are given by $x_{w_1}\cdots x_{w_k}$ and $x_{z_1}\cdots x_{z_k}$, respectively, and these are different, we conclude that $G_{n,k}^{\sigma_1} \neq G_{n,k}^{\sigma_2}$.

Now, assume that $Y$ is not empty. Let the next block in the ranking induced by $\prec_1$ be $\{w_1,w_2\}$ and the corresponding block for $\prec_2$ be $\{z_1,z_2\}$. We know that $\{w_1,w_2\} \neq \{z_1,z_2\}$. By Theorem~\ref{thm:GBasis}, there exists a degree $d$ (determined by $|Y|$) such that the leading terms $\ini(g_{A,n,k})$ of degree $d$ in $G_{n,k}^{\sigma_1}$ are divisible by $w_1 w_2$ and otherwise only by variables in $Y$. In $G_{n,k}^{\sigma_2}$, a similar statement holds, with $z_1 z_2$ replacing $w_1 w_2$. Thus, the degree $d$ squarefree components of $\ini_{\prec_1}(I_{n,k})$ and $\ini_{\prec_2}(I_{n,k})$ differ, hence $G_{n,k}^{\sigma_1} \neq G_{n,k}^{\sigma_2}$, which completes the proof.
\end{proof}

\begin{remark}\label{rem:InitialIdealsBettiNumbers}
For any given term ordering, Lemma~\ref{lem:Invariance_G_n_k} states that the ordering $x_{\sigma(1)} \succ \cdots \succ x_{\sigma(n)}$ for some permutation $\sigma \in S_n$ uniquely determines the Gröbner basis of $I_{n,k}$ for that ordering. In fact, knowledge of the induced block variable ranking is sufficient. Therefore, all Gröbner bases of $I_{n,k}$ take the form $G_{n,k}^{\sigma}$ for some permutation $\sigma$, and all initial ideals of $I_{n,k}$ are equivalent up to renumbering of the variables. This implies that they share the same (multigraded) Betti numbers.
\end{remark}

\begin{example}\label{ex:N_n_k_for_n=5,6,7_k=2,3}
The number of distinct reduced Gröbner bases for the ideals $I_{n,k}$, denoted by $N_{n,k}$, is given in some examples as 
    \begin{align*}
        N_{5,2}=&\binom{5}{2,2,1}
        =30,\quad\quad\ 
        N_{6,2}=\binom{6}{2,2,2}
        =90,\quad
        N_{7,2}=\binom{7}{2,2,2,1}=
        630,\\
        N_{5,3}=&\binom{5}{3,2}
        =10,\quad\quad\quad
        N_{6,3}=\binom{6}{3,2,1}
        =60,\quad
        N_{7,3}=\binom{7}{3,2,2}=
        210.\\
    \end{align*}
\end{example}

\section{The Lefschetz properties}

We now apply the results from the previous section to study the Lefschetz properties. Recall that an Artinian standard graded algebra $A$ has the weak Lefschetz property (WLP) if there exists a linear element $\ell$ such that the multiplication map $\cdot \ell: A_i \to A_{i+1}$ has maximal rank for all $i$, that is, the map is either surjective or injective. Similarly, $A$ has the strong Lefschetz property (SLP) if there exists a linear element $\ell$ such that the multiplication map $\cdot \ell^k: A_i \to A_{i+k}$ has maximal rank for all $i$ and $k$. For an introduction to the Lefschetz properties, see \cite{book, atour}.

\subsection{A new proof of the fact that the squarefree algebra has the SLP}

As a first application of our Gr\"obner basis results to Lefschetz properties, we can show that the squarefree algebra has the SLP.

\begin{theorem}
The squarefree algebra has the SLP, and the algebra $R/I_{n,k}$ is thin. 
\end{theorem}

\begin{proof}
Recall that  an algebra defined by an ideal of forms of degree $a_1,\dots, a_m$ is thin if it has a Hilbert series given by
$$\left[\frac{\prod_{i=1}^{m}(1-t^{a_{i}})}{(1-t)^n}\right].$$ 
Moreover, the squarefree algebra $R/(x_1^2,\dots, x_n^2)$ has the strong Lefschetz property if and only if the algebras $R/I_{n,k}$ are thin for all $k$. That is, we want the Hilbert series of $R/I_{n,k}$ to be given by
\[
\left[\frac{\prod_{i=1}^{n}(1-t^{2})}{(1-t)^n}(1-t^k)\right] = \left[(1+t)^n(1-t^k) \right].
\]
However, this is precisely the statement of Corollary \ref{cor:Hilbert_series}. Therefore, all algebras $R/I_{n,k}$ are thin, and the squarefree algebra has the SLP.
\end{proof}

As mentioned in the introduction, the SLP for an Artinian monomial complete intersection can be derived from the squarefree algebra. The argument given by Hara and Watanabe in \cite{Watanabe_Boolean}
proceeds as follows. Suppose we wish to show that the algebra
$$B=\mathbf{k}[y_1,\dots, y_m]/(y_1^{d_1+1},\dots, y_m^{d_m+1})$$ has the SLP. The idea is then to embed $B$ into the squarefree algebra $$A=\mathbf{k}[x_1,\dots, x_n]/(x_1^2,\dots, x_n^2)$$ where $n=d_1+\cdots + d_m$ is chosen so that $A$ and $B$ have the same socle degree. Now define a map $\phi:\mathbf{k}[y_1,\dots, y_m]\to A$ by \[
\phi(y_i) = x_{d_1+d_2+\cdots + d_{i-1} + 1} + x_{d_1+d_2+\cdots + d_{i-1} + 2} + \cdots + x_{d_1+d_2+\cdots + d_{i-1} + d_i}.
\]
One can then verify that $\ker(\phi)=(y_1^{d_1+1},\dots, y_m^{d_m+1})$, so $\phi$ induces an inclusion of $B$ into $A$. Thus, if $\ell=y_1+\cdots + y_r$ is a linear form on $B$, then $\phi(\ell)=x_1+\cdots + x_n = \ell'$ is a linear form on $A$ where we know that $\cdot (\ell')^{n-2i}:A_i \to A_{n-i}$ is bijective. The  commutative diagram in Figure \ref{fig:com_diagram} 
\begin{figure}[h]
\begin{center}   
\begin{tikzcd}
A_i \arrow[r, "(\ell')^{n-2i}"]                      & A_{n-i}                         \\
B_i \arrow[u, "\phi", hook] \arrow[r, "\ell^{n-2i}"] & B_{n-1} \arrow[u, "\phi", hook]
\end{tikzcd}
\end{center}
\caption{The commutative diagram used to show that $A$ has the SLP.}    
\label{fig:com_diagram}
\end{figure}
then forces $\cdot \ell^{n-2i}:A_i \to A_{n-i}$ to be injective. Hence, it has full rank and $A$ has the SLP.

While the SLP for Artinian monomial complete intersections was established by Stanley and Watanabe in the 1980s, the corresponding result for squarefree algebras dates back to the 1970s, originally studied in the context of combinatorial properties of zero-one matrices. Fröberg and Hollman \cite{Hollman} credit this earlier work to Kantor \cite{Kantor}, Graver and Jurkat \cite{Graver}, and Wilson \cite{Wilson}, with these references initially brought to their attention by Bernt Lindström.

Besides the results from the 1970s, we are aware of the proof by Hara and Watanabe for the case $d=2$, an argument by Ikeda \cite{Ikeda} from the 1990s, and a more recent proof by Phuong and Tran \cite{Phuong}. Therefore, to the best of our knowledge, our proof represents the seventh contribution to the specific case of $d=2$.

\subsection{The WLP for powers of general linear forms in the squarefree algebra}
Let $R=\kb[x_1,\ldots,x
_n]$. Given that $R/(\ell_1^{a_1},\ldots,\ell_{n+1}^{a_{n+1}}) $ is a thin algebra for general $\ell_i$, it is natural to ask if $R/(\ell_1^d,\ldots,\ell_{n+2}^d)$ is also thin. 
However, Fr\"oberg and Hollman \cite{Hollman} showed that this is not always the case. They showed this by computing the Hilbert series for the case $n=5,a_1=\ldots=a_7=2$, which equals
$1+ 5t +8t^2+t^3$, differing from the expected series $\left[\frac{(1-t^2)^7}{(1-t)^5)}\right] = 1 + 5t + 8t^2$. This discrepancy was established through a symbolic computation in the algebra $$\mathbf{k}[x_1,\ldots,x_5]/(x_1^2,x_2^2,x_3^2,x_4^2,x_5^2,(x_1+\cdots+x_5)^2, (a_1x_1+\cdots+a_5x_5)^2),$$ which was a computationally intensive task given the available resources in 1994.

In 2000, Cruz and Iarrobino \cite{CruzIarrobino} conjectured that the socle degree for the case $a_1=\cdots=a_{n+2}=2$ equals $n/2$ when $n$ is even, and $(n+1)/2$ when $n$ is odd. Moreover,  they conjectured that the dimension in the socle equals $2^{n/2}$ in the case of an even number of variables, and $1$  in the odd case. This conjecture  was settled in 2008 by Sturmfels and Xu \cite{SturmfelsXu}.  It is straightforward to show that their result implies that 
$R/(\ell_1^{2},\ldots,\ell_{n+2}^{2})$ is thin if and only if $n \in \{1,2,3,4,6\}$. Thus $R/(\ell_1^{2},\ldots,\ell_{n+1}^{2})$ has the WLP if and only if $n \in \{1,2,3,4,6\}$.

In 2012, Migliore, Miro-R\'oig, and Nagel \cite{Miglioreetal} showed that when $n\geq4$ is even, the algebra  $R/(\ell_1^d,\ldots,\ell_{n+1}^d)$ fails the WLP for all $d$ except for $(n,d) = (4,2)$, and conjectured that it should fail the WLP when $n \geq 9$ is odd. The conjecture was settled by Boij and the second author in 2023 \cite{boijlundqvist}, although many special cases had been settled earlier than that, for instance the case $d$ large enough by Nagel and Trok \cite{nageltrok}. The proof relies on Macaulay's inverse system and a reduction to the case $d=2$, from which it is shown that the WLP fails because of surjectivity.

In this section we make an attempt on the non-equigenerated case and give a complete characterization of $n,k$ for which the algebra $R/(\ell_1^2,\dots, \ell_n^2, \ell_{n+1}^k)$ defined by powers of general linear forms $\ell_1,\dots, \ell_{n+1}$ has the WLP.

Also, in this case, it turns out that the question of the WLP can be reduced to examining the multiplication map by $\ell_{n+2}^k$ on $R/(\ell_1^2, \dots, \ell_n^2, \ell_{n+1}^2)$. However, contrary to the result for the equigenerated case, failure of the WLP is due to the lack of injectivity, not surjectivity. In particular, this means that we must find non-trivial elements in the kernel, which turns out to be technically challenging. As a by-product, we obtain non-trivial identities in the squarefree algebra, which are interesting in their own right (see Proposition \ref{prop:syzygy_relation} and Proposition \ref{prop:syz_relation_even}).

\medskip
The classification result that we will prove is stated in the following theorem, whose proof will be presented in the end of this section.

\begin{theorem}\label{thm:WLP_linear_forms}
Let $\mathbf{k}$ be a field of characteristic zero and fix integers $n,k\geq 2$. Then for general linear forms $\ell_1,\dots, \ell_{n+1}$, the algebra $\mathbf{k}[x_1,\dots, x_n]/(\ell_1^2,\dots, \ell_n^2, \ell_{n+1}^k)$ has the weak Lefschetz property if and only if \[
\begin{cases}
k\geq \frac{n-3}{2} & \text{ for } n \text{ odd,}\\
k\geq \frac{n}{2} & \text{ for } n \text{ even.}\\
\end{cases}
\] 
\end{theorem}

\begin{remark}\label{rem:equiv}
\begin{itemize}
\item[{\rm (i)}] As noted earlier, a change of variables gives the isomorphism
\begin{equation*}\label{eq:simplification}
    \mathbf{k}[x_1, \ldots, x_n]/(\ell_1^{2}, \ldots, \ell_n^2, \ell_{n+1}^{k}) \cong 
    \mathbf{k}[x_1, \ldots, x_n]/(x_1^{2}, \ldots, x_n^{2}, (x_1 + \cdots + x_n)^{k}).
\end{equation*}
Thus, Theorem~\ref{thm:WLP_linear_forms} provides a characterization of the integers $n$ and $k$ for which the ideal $I_{n,k}$ defines an algebra with the weak Lefschetz property.

\item[{\rm (ii)}] Note that multiplication by a form $f$ of degree $d$ on an algebra $A$ has full rank if and only if 
\[
\mathrm{HS}(A/(f);t) = [(1 - t^d) \mathrm{HS}(A;t)],
\]
where the square brackets indicate truncation of the polynomial at the first non-positive coefficient.
\end{itemize}
\end{remark}

In the following lemmas, it is sometimes useful to work with the ideal $I_{n-1,2}$ in place of $I_{n,k}$ to streamline the analysis and simplify computations.

\begin{lemma}\label{lem:full_rank_switch}
The algebra $\mathbf{k}[x_1,\ldots,x_n]/(x_1^{2},\ldots,x_n^{2}, (x_1+\cdots+x_n)^{k})$ has the WLP if and only if, for a general linear form $\ell$ in $\mathbf{k}[x_1,\ldots,x_{n-1}]$, the multiplication maps by $\ell^k$ yield full rank maps on $A=\mathbf{k}[x_1,\ldots,x_{n-1}]/(x_1^{2},\ldots,x_{n-1}^{2}, (x_1+\cdots+x_{n-1})^2)$.
\end{lemma}

\begin{proof}
For a general linear form $\ell'$ in $n$ variables, we have that
\begin{align*}
&\mathbf{k}[x_1,\ldots,x_n]/(x_1^{2},\ldots,x_{n}^{2}, (x_1+\cdots+x_n)^{k}, \ell') \\
&= \mathbf{k}[x_1,\ldots,x_n]/(x_1^{2},\ldots,x_{n}^{2}, (x_1+\cdots+x_n)^{k}, a_1 x_1 + \cdots + a_n x_n)\\ 
&\cong \mathbf{k}[x_1,\ldots,x_n]/(x_1^{2},\ldots,x_{n}^{2}, (a_1' x_1 + \cdots + a_n' x_n)^{k}, x_1+\cdots+x_n)\\ 
&\cong \mathbf{k}[x_1,\ldots,x_{n-1}]/(x_1^{2},\ldots,x_{n-1}^{2}, (x_1+\cdots+x_{n-1})^2, (a_1'' x_1 + \cdots + a_{n-1}'' x_{n-1})^{k})
\end{align*}
for some generic coefficients $a_i, a_i'$ and $a_i''$. We denote the last algebra as $B$. 
Since we know that 
$\mathrm{HS}(\mathbf{k}[x_1,\ldots,x_n]/I_{n,k};t)=[(1-t^k)(1+t)^n]$, 
we conclude that multiplication by $\ell^k$ on the algebra $A$ 
has full rank if and only if
\begin{align*}
\mathrm{HS}(B;t) &= [(1-t^k)\mathrm{HS}(A
;t)] 
=[(1-t^k)(1-t^2)(1+t)^{n-1}] \\
&=[(1-t)(1-t^k)(1+t)^{n}]
=[(1-t)\mathrm{HS}(\mathbf{k}[x_1,\ldots,x_n]/I_{n,k};t)].
\end{align*}
However, $B$ has this Hilbert series if and only if multiplication by a general linear form $\ell'$ has full rank on $\mathbf{k}[x_1, \ldots, x_n] / I_{n,k}$. We conclude that $\mathbf{k}[x_1, \ldots, x_n] / I_{n,k}$ has the WLP if and only if multiplication by $\ell^k$ has full rank on $A$, 
as required.
\end{proof}

In the upcoming lemmas, we will need some notation. For a set $S=\{s_1,\dots, s_p\}$, we denote $a_S=a_{s_1}\cdots a_{s_p}$ and $x_S=x_{s_1}\cdots x_{s_p}$.
For any $(k-1)$ subset $T\subseteq[n]$, we define $\mathcal{T}_i=\{S\subseteq [n]:\ |S|=|T|=k-1,\ |S\cap T|=i\}$. 

 \medskip
 
The following proposition is crucial for establishing the failure of the WLP.
\begin{proposition}\label{prop:syzygy_relation}
Let $n = 2p + 1$ 
for some $p \geq 1$, and let $\ell = a_1 x_1 + \cdots + a_n x_n$ be a general linear form. Then there exists another linear form $\ell'$ and a degree $(p-1)$ form $g$ such that in $\mathbf{k}[x_1, \ldots, x_n] / (x_1^2, \ldots, x_n^2)$, we have that
\[
\ell^p \ell' = (x_1 + \cdots + x_n)^2 g.
\]
\end{proposition}

\begin{proof}
Let $B=\mathbf{k}[x_1, \ldots, x_n] / (x_1^2, \ldots, x_n^2)$. In $B$, we have
\[\ell^p = p!\sum_{T=\{i_1<\dots <i_p\}}a_Tx_T\quad\text{and}\quad (x_1+\cdots + x_n)^2 = 2\sum_{i<j}x_ix_j
\]
We will proceed using these expressions. We claim that we can choose 
\[{\ell'=\sum_{i=1}^{n}\big(pa_i -\sum_{\substack{j=1, j\neq i}}^{n}a_j\big)a_ix_i}\quad\text{and}\quad g=
\sum_{\substack{T\subseteq [n]\\ |T|=p-1}}\big(\sum_{i=0}^{p-1}b_i\sum_{\substack{S\in \mathcal{T}_i}}\frac{a_1\cdots a_n}{a_S}\big)x_T,\] 
where the coefficients $b_i$ are determined as solutions to the system~\eqref{eqn:syz_system} that will be determined later. First, we observe that in $A$, we can write $\ell^p\ell'$ as 
\begin{align*}
\ell^p\ell' &= \big(p!\sum_{T=\{i_1<\dots <i_p\}}a_T x_T 
\big) \big(\sum_{i=1}^{n}\big(pa_i -\sum_{\substack{j=1,j\neq i}}^{n}a_j\big)a_ix_i\big) \\
&= -(p+1)!\sum_{T=\{i_1<\dots <i_{k+1}\}}a_Tx_T\   
\big(\sum_{j\notin T 
}
a_j \big).
\end{align*}
The equality holds because each squarefree monomial $x_T$ of degree $p+1$ is formed such that, for any $x_s$ in its support, $a_s$ appears in 
$\textstyle{pa_i -\sum_{\substack{j=1, j\neq i}}^{n}a_j}$  
once with coefficient $p$, and $p$ times with coefficient $-1$. For any $x_s$ not dividing $x_T$,~$a_s$~appears $(p+1)$ times with coefficient $-1$, resulting in the final coefficient $-(p+1) \sum_{j \notin T} a_j$.

\medskip

Now, we verify that $g$ satisfies the desired condition. To see this, observe that
\begin{align*}
\sum_{T=\{i_1<\dots <i_{p+1}\}}a_Tx_T\ \big(\sum_{j\notin  T}a_j \big)=\sum_{T=\{i_1<\dots <i_{k+1}\}}\big(\!\!\!\!\sum_{\substack{S\subseteq [n]\\ |S\cap T|=0 \\ |S|=k-1}
}\frac{a_1\cdots a_n}{a_S} \big)x_T.
\end{align*}
Thus, in $g(x_1 + \cdots + x_n)^2$, we observe that to find the coefficient of $x_T$ for $T=\{i_1, \ldots i_{p+1}\}$, we need all the coefficients $(a_1 \cdots a_n)/{a_S}$ arising from sets $S$ with $|S \cap T| \neq 0$ to cancel out. Fix a set $S$ in $\mathcal{T}_s$ for  
some $s > 0$, where $|S| = k-1$ and $|S \cap T| = s$. The coefficient $(a_1 \cdots a_n)/{a_S}$ appears in $g$ with coefficient $b_i$ for any monomial $x_{T'}$ arising from a set $T'$ with $|T'| = p-1$, $T' \subseteq T$ (so that $x_{T'}$ divides $x_T$), and $|S \cap T'| = i$. The number of such sets $T'$ is given by
\[
\binom{s}{i} \binom{p+1-s}{p-1-i},
\]
where the first binomial coefficient counts the possible intersections of $S$ and $T'$, while the second counts the possible remaining elements of $T'$. Hence, for $g(x_1 + \cdots + x_n)^2$ to be a constant multiple of $\ell^p \ell'$, we must solve the system of equations
\begin{equation}\label{eqn:syz_system}
\sum_{i=0}^{p-1}\binom{s}{i}\binom{p+1-s}{p-1-i}b_i=0
\end{equation}
for $s = 1, 2, \dots, p-1$. This is a triangular system, as the second binomial coefficient is non-zero only for $i = s-2, s-1, s$, so $b_i$ has non-zero coefficients only for these values of $i$. Therefore, the system has a solution. As a result, there exists a constant $c$ such that for $g' = c g$, we have $\ell^p \ell' = g'(x_1 + \cdots + x_n)^2$, as desired.
\end{proof}

\begin{remark}
One can determine the polynomial $g$ in Proposition~\ref{prop:syzygy_relation} explicitly. By imposing the condition that $\binom{p+1}{p-1} b_0 = 1$, the system~\eqref{eqn:syz_system} becomes a special case of the system encountered in \eqref{eq:InhomogLinSys_3}, whose solution is provided in Lemma~\ref{lem:BinomialCoeffsAlternatingSums}.
\end{remark}

\begin{example}
Let $n = 5$ and $p = 2$. Proposition~\ref{prop:syzygy_relation} then concerns the forms
\begin{align*}
\ell &= a_1 x_1 + a_2 x_2 + a_3 x_3 + a_4 x_4 + a_5 x_5, \\
\ell' &= (2a_1 - a_2 - a_3 - a_4 - a_5) a_1 x_1 + \cdots + (2a_5 - a_1 - a_2 - a_3 - a_4) a_5 x_5, \\
g &= \sum_{j=1}^{5} \big(b_1 \frac{a_1 a_2 a_3 a_4 a_5}{a_j} + b_0 \sum_{\substack{p=1 \\ p \neq j}}^{5} \frac{a_1 a_2 a_3 a_4 a_5}{a_p}\big) x_j,
\end{align*}
where $2b_0 + b_1 = 0$. If we choose $b_0 = -1$ and $b_1 = 2$, a simple calculation shows
\begin{align*}
\ell^2 \ell' &= g (x_1 + x_2 + x_3 + x_4 + x_5)^2 \\
&= -6 \left(a_1 a_2 a_3 x_1 x_2 x_3 (a_4 + a_5) + \cdots + a_3 a_4 a_5 x_3 x_4 x_5 (a_1 + a_2)\right)
\end{align*}
in the algebra $\mathbf{k}[x_1, \ldots, x_5] / (x_1^2, \ldots, x_5^2)$. This relation for $n = 5$ and $p = 2$ was the one that was computed symbolically by Fröberg and Hollman as mentioned earlier.
\end{example}

Proposition~\ref{prop:syzygy_relation} will be used to establish the failure of the WLP for even $n$ in Theorem~\ref{thm:WLP_linear_forms}. For the odd $n$, we need the following proposition.

\begin{proposition}\label{prop:syz_relation_even}
Let $n=2p$ for some $p\geq 3$ and let $\ell=a_1x_1 + \cdots + a_nx_n$ be a general linear form. Then there is a degree $2$ form $f$ and a degree $p-2$ form $g$ such that
\[
\ell^{p-2}f = (x_1+\cdots + x_n)^2g
\]
in $\kb[x_1,\dots, x_n]/(x_1^2,\dots, x_n^2)$ and $f$ is not a multiple of $(x_1+\cdots + x_n)^2$.
\end{proposition}

\begin{proof}
The proof follows similarly to that of Proposition~\ref{prop:syzygy_relation}.
We claim that 
\[f= \sum_{\substack{T\subseteq [n]\\ |T|=2}}a_Tx_T\ \big( \sum_{i=0}^{2}\lambda_i \sum_{\substack{S\in\mathcal{T}_i}}a_S \big)\quad\text{and}\quad g = \sum_{\substack{T \subseteq [n] \\ |T| = p-2}} x_T\ \big( \sum_{i=0}^{p-2} b_i \sum_{\substack{S \in \mathcal{T}_{p-2}}} \frac{a_1 \cdots a_n}{a_S} \big),\] 
where the coefficients $\lambda_i$ and $b_i$ are determined as solutions to the systems~\eqref{coef:lambda} and \eqref{coeff:b_i}, respectively. 
We first show that $\ell^{p-2}f$ can be written as
\begin{equation}\label{eq:lambda_syz}
\ell^{p-2}f = c\sum_{\substack{T\subseteq [n]\\ |T|=p}}a_Tx_T\ \big(\! \sum_{\substack{S\in\mathcal{T}_0}}a_S \big) \\
= c\sum_{\substack{T\subseteq [n]\\ |T|=p}}x_T\ \big( \!\!\!\!\sum_{\substack{S\subseteq [n]\\ |S\cap T|=0\\ |S|=p-2}}\frac{a_1\cdots a_{n}}{a_S} \big)
\end{equation}
for some constant $c$.
The above equalities hold because, in $\ell^{p-2}f$, the coefficient in front of $x_T$ for a set $T$ with $|T| = p$ is given by $a_T$, multiplied by sums of terms of the form $a_S$ for sets $S$ of size 2. We want to show that all constants corresponding to such sets with $|S \cap T| \neq 0$ cancel out. Consider a set $S$ with $|S| = 2$ and $|S \cap T| = s > 0$. Then $a_S$ appears in $f$ with coefficient $\lambda_i$ for any monomial $x_{T'}$, where $T'$ is a set with $|T'| = 2$, $T' \subseteq T$, and $|S \cap T'| = i$. Since there are 
$\binom{s}{i} \binom{p-s}{2-i}$
such sets $T'$, the system of equations that needs to be solved is
\begin{equation}\label{coef:lambda}
\begin{cases}
\lambda_2 + \binom{2}{1}\binom{p-2}{1}\lambda_1 + \binom{p-2}{2}\lambda_0 &= 0, \\
\binom{p-1}{1} \lambda_1 + \binom{p-1}{2} \lambda_0 &= 0.
\end{cases}
\end{equation}
Since this system is triangular, it has a solution, and we can find $\lambda_0$, $\lambda_1$, and $\lambda_2$ such that Equation \eqref{eq:lambda_syz} is satisfied.

\medskip
We next want $g(x_1 + \cdots + x_n)^2$ to be given by the expression in equation \eqref{eq:lambda_syz}. By the same arguments as in Proposition~\ref{prop:syzygy_relation}, and by considering how many times sets with specified intersection sizes occur, $g(x_1 + \cdots + x_n)^2$ will be in the desired form if the system given by
\begin{equation}\label{coeff:b_i}
    \sum_{i=0}^{p-2} \binom{s}{i} \binom{p-s}{p-i-2} b_i = 0
\end{equation}
for $s = 1, 2, \dots, p-2$, has a solution. This system is triangular since only the $b_i$ with $i = s$, $i = s-1$, and $i = s-2$ have non-zero coefficients. Hence, the system has a solution, and our claimed form of $g$ (or possibly $cg$ for some constant $c$) satisfies that $\ell^{p-2} f = (x_1 + \cdots + x_n)^2 g$. Finally, it is clear that $f$ is not a multiple of $(x_1 + \cdots + x_n)^2$ in $k[x_1, \dots, x_n]/(x_1^2, \dots, x_n^2)$ since $f$ has different coefficients in front of its non-zero terms, while $(x_1 + \cdots + x_n)^2$ does not.
\end{proof}

While Propositions~\ref{prop:syzygy_relation} and~\ref{prop:syz_relation_even} will be used to establish failure of the WLP, the next lemma will be used to lift this failure to more degrees and address the nontrivial cases for which the WLP holds.

\begin{lemma}\label{lem:Injective_WLP}
Let $B=\mathbf{k}[x_1,\ldots,x_n]/(x_1^{2},\ldots,x_{n}^{2}, (x_1+\cdots+x_n)^{2})$ and let $\ell$ be a general linear form in $B$. If $n\geq 2p+2$, then $\cdot \ell^p : B_1 \to B_{p+1}$ is injective.
\end{lemma}

\begin{proof}
We begin by noting that, by an argument of Conca \cite[Theorem 1.1]{Conca}, which was first pointed out by Wiebe \cite[Proposition 2.9]{Wiebe} to hold in this generality, it is sufficient to demonstrate that the map $\cdot \ell^p : A_1 \to A_{p+1}$ has full rank, where $A = \mathbf{k}[x_1, \ldots, x_n]/\ini(I_{n,2})$. Next, since $A$ is a monomial algebra, it suffices to consider the linear form $\ell = x_1 + \dots + x_n$, as established for the WLP by Migliore, Mir\'o-Roig, and Nagel in \cite[Proposition 2.2]{Miglioretrans} and generalized to the SLP and arbitrary characteristics by Nicklasson \cite[Theorem 2.2] {Nicklasson_SLP_two_variables}. We will prove that the map is injective by showing that if
\[
(x_1+\cdots + x_{n})^p \big(\sum_{i=1}^{n}c_ix_i \big) = 0
\]
in $A$, then all $c_i=0$. To this end, we expand and observe that
\[
(x_1+\cdots + x_{n})^p \big(\sum_{i=1}^{n}c_ix_i \big) = p!\sum_{\substack{S\subseteq [n]\\ |S|=p+1}}\big(\sum_{j\in S}c_j\big)x_S.
\]
By our restriction on $n$ and $p$, it follows from our description of $\ini(I_{n,2})$ in Theorem~\ref{thm:GBasis} that all monomials of the form $x_ix_{n-p+1}x_{n-p+2}\cdots x_{n}$ for $i=1,2,\dots, n-p$ are nonzero in $A_{p+1}$. Therefore, we have 
\[
c_i + \sum_{j=n-p+1}^{n} c_j = 0,
\]
which implies that $c_1 = c_2 = \cdots = c_{n-p}$. Additionally, the monomials
\[
x_{n-p-1} x_{n-p} \cdots x_{n-p+i-2} x_{n-p+i} \cdots x_n
\]
are nonzero in $A_{p+1}$ for $i = 1, 2, \dots, p+1$. Hence, we obtain
\[
\sum_{\substack{j=n-p-1 \\ j \neq n-p+i-1}}^{n} c_j = 0,
\]
which, combined with the relation $c_{n-p} + \cdots + c_n = 0$, implies that $c_{n-p-1} = c_{n-p} = \cdots = c_n$. Therefore, all the $c_j$'s must be equal, and since $c_{n-p} + \cdots + c_n = 0$, it follows that $c_j = 0$ for all $j$. This establishes the injectivity of the map.
\end{proof}

Similarly to how Propositions~\ref{prop:syzygy_relation} and~\ref{prop:syz_relation_even} cover both even and odd cases, an alternative version of Lemma~\ref{lem:Injective_WLP} is also necessary.
\begin{lemma}\label{lem:Lift_kernel_deg2}
Consider the algebra $B = \mathbf{k}[x_1, \dots, x_n] / (x_1^2, \dots, x_n^2, (x_1 + \cdots + x_n)^2)$, where $n = 2p$ and $p \geq 3$. Then, for $q$ such that $2 < q < p$ and $\ell$ a general linear form, the map $\cdot \ell^{q-2}: B_2 \to B_q$ is injective.
\end{lemma}

\begin{proof}
As in Lemma~\ref{lem:Injective_WLP}, it suffices to show that the map $\cdot \ell^{q-2}:A_2 \to A_{q}$ is injective, where $A=\mathbf{k}[x_1,\dots, x_n]/\ini(I_{n,2})$ and $\ell=x_1+\cdots + x_n$. Let 
\[
g=\sum_{i<j}c_{i,j}x_ix_j
\]
be a quadratic form. Then
\begin{equation}\label{eq:deg2_prod}
g(x_1+\cdots + x_n)^{q-2} = (q-2)!\sum_{|T|=q}\big(\!\!\!\sum_{\{i,j\}\in T}c_{i,j} \big)x_T
\end{equation}
in $A$. We aim to show that if the expression in \eqref{eq:deg2_prod} equals zero, then all $c_{i,j} = 0$. By the bijection in Lemma~\ref{lem:LatticePathBijections}, it is clear that only the monomials corresponding to lattice paths that do not intersect the line $y = x + 2$ impose conditions on the coefficients. Focusing on the terms in \eqref{eq:deg2_prod} involving variables $x_i$ with $i \geq q$, the corresponding paths do not intersect the line and therefore provide constraints on the coefficients. These are precisely the equations that arise from multiplying by $(\ell')^{q-2}$ for $\ell' = x_q + \cdots + x_n$ in the algebra $B' = \mathbf{k}[x_k, \dots, x_n] / (x_q^2, \dots, x_n^2)$. Since $B'$ is a monomial complete intersection with the strong Lefschetz property, we know that $\cdot (\ell')^{q-2}: B_2' \to B_q'$ has full rank. Given that the socle degree of $B'$ is $n - (q-1) = 2p - q + 1$, we have $\dim_{\mathbf{k}} B_q' \geq \dim_{\mathbf{k}} B_2'$ because $2(p-q) + 1 \geq 2$. Thus, $\cdot (\ell')^{q-2}: B_2' \to B_q'$ is injective, and hence $c_{i,j} = 0$ for all $i, j \geq q$.

To prove that all $c_{i,j} = 0$ without restrictions on $i, j$, we will show that no smallest index $r$ exists for which $c_{i,j} = 0$ whenever $i, j \geq r$. Suppose, for contradiction, that $c_{i,j} = 0$ for all $i, j \geq r > 1$. We will demonstrate that this implies $c_{i,j} = 0$ for all $i, j \geq r - 1$, as well.
Considering all terms in \eqref{eq:deg2_prod} involving variables with index at least $r-1$, we obtain, by assumption, the expression
\[
\sum_{\substack{T \subset \{r, \dots, n\} \\ |T| = q-1}} \left(\sum_{j \in T} c_{r-1,j} \right) x_{r-1} x_T.
\]
Factoring out $x_{r-1}$, this expression resembles the one in Lemma~\ref{lem:Injective_WLP}. Similar to that case, examining monomials of the form $x_T = x_i x_{n-q+3} x_{n-q+4} \cdots x_n$ for $i = r, r+1, \dots, n-q+2$ yields
\[
c_{r-1, i} + \sum_{j=n-q+3}^n c_{r-1, j} = 0,
\]
which implies $c_{r-1, r} = \cdots = c_{r-1, n-q+2}$. Additionally, the monomials
\[
x_{n-q+1} x_{n-q+2} \cdots x_{n-q+j-1} x_{n-q+j+1} \cdots x_n
\]
for $j = 1, 2, \dots, q$ imply that $c_{r-1, n-q+1} = c_{r-1, n-q+2} = \cdots = c_{r-1, n}$. Thus, all $c_{r-1, j}$ are equal, and any one of these equations implies that $c_{r-1, j} = 0$ for all $j \geq r$, as desired.
Note that the case $r = 2$ is slightly different since $x_1 x_2 x_{n-q+3} \cdots x_n = 0$ in $A$, which means we obtain no information about $c_{1,2}$. However, as $x_1 x_2 = 0$ in $A$, this poses no issue. Thus, the proof is complete.
\end{proof}

With this preparation, we are now ready to prove Theorem~\ref{thm:WLP_linear_forms}.
\begin{proof}[{\bf Proof of Theorem~\ref{thm:WLP_linear_forms}}]
By Remark~\ref{rem:equiv}(i), it suffices to prove the statement for the algebra $\mathbf{k}[x_1, \ldots, x_n]/(x_1^{2}, \ldots, x_{n}^{2}, (x_1 + \cdots + x_n)^k)$. By Lemma~\ref{lem:full_rank_switch}, this is equivalent to showing that the $k$-th power of a general linear form has full rank on $\mathbf{k}[x_1, \ldots, x_{n-1}]/(x_1^{2}, \ldots, x_{n-1}^{2}, (x_1 + \cdots + x_{n-1})^2)$ for the same values of $n$ and $k$. We begin by examining the ideals for which we want to show the WLP.

Let $A = \mathbf{k}[x_1, \ldots, x_{n-1}]/(x_1^{2}, \ldots, x_{{n-1}}^{2}, (x_1 + \cdots + x_{n-1})^{2})$. The socle degree of $A$ is given by the smallest integer $d$ for which $\binom{n-1}{d} - \binom{n-1}{d-2} > 0$, which is $(n-1)/2$ if $n$ is odd, and $n/2$ if $n$ is even. Since any multiplication by a power of a linear form from $A_0$, or into a zero-dimensional vector space, has full rank, it follows that any $k$ greater than $(n-1)/2$ (for $n$ odd) or greater than $n/2$ (for $n$ even) works. Additionally, when $n = 2k + 3$, the socle degree is $k + 1$, so the only map potentially lacking full rank is from degree 1 to degree $k + 1$. However, this map is injective by Lemma~\ref{lem:Injective_WLP}, establishing all the required full-rank maps.

We are now left to show that the remaining values of $n,k$ give algebras that do not have the WLP. Assume first that $n=2p+2$ is even. Then we claim that for any general linear form $\ell$ in $A$, the map 
\[
\cdot \ell^k:A_{p+1-k} \to A_{p+1}
\]
does not have full rank. For $k = p$, Proposition~\ref{prop:syzygy_relation} provides a linear form $\ell' \in A_1$ such that $\ell^p \ell' = 0$ in $A$. Additionally, by Lemma~\ref{lem:Injective_WLP}, we have $\ell^k \ell' \neq 0$ for any $k < p$. Therefore, $\ell^{p-k} \ell'$ is a nonzero element in the kernel of $\ell^k$ for all $k = 2, \dots, p-1$, demonstrating that $\cdot \ell^k: A_{p+1-k} \to A_{p+1}$ cannot be injective. 
Moreover, this map cannot be surjective either. Using the theory of inverse systems, Boij and the second author in \cite[Theorem 5.2]{boijlundqvist} have shown that $A/(\ell^k)$ is always nonzero in degree $p + 1$ for any $k \geq 2$. This establishes all required failures of maximal rank for even $n$.

The case where $n = 2p + 1$ is odd follows a similar proof. First, the map 
\[
\cdot \ell^k: A_{p-k} \to A_p
\]
is not injective for $k = p - 2$ by Proposition~\ref{prop:syz_relation_even}, as there exists a nonzero form $f$ of degree 2 with $\ell^{p-2} f = 0$ in $A$. Furthermore, by Lemma~\ref{lem:Lift_kernel_deg2}, we have $\ell^{p-k-2} f \neq 0$ in $A$ for $k < p - 2$, which means that $\ell^{p-k-2} f$ is a nonzero element in the kernel of $\cdot \ell^k: A_{p-k} \to A_p$, demonstrating that this map is never injective.
Additionally, by 
\cite[Theorem 5.2]{boijlundqvist}, 
setting the last variable to zero shows that $A/(\ell^k)$ is nonzero in degree $p$ for any $k \geq 2$. This completes the verification of all required failures of full rank, proving the theorem.
\end{proof}

\subsection{A remark on the Fr\"oberg conjecture}
We have the isomorphism $$\mathbf{k}[x_1,\dots, x_n]/(\ell_1^2,\dots, \ell_n^2, \ell_{n+1}^k,\ell_{n+2}) \cong 
\mathbf{k}[x_1,\dots, x_{n-1}]/(\ell_1^2,\dots, \ell_n^2, \ell_{n+1}^k),$$
which implies, by Theorem~\ref{thm:WLP_linear_forms}, that the latter algebra is thin when $$
\begin{cases}
k\geq \frac{n-3}{2} & \text{ for } n \text{ odd,}\\
k\geq \frac{n}{2} & \text{ for } n \text{ even}.\\
\end{cases}
$$
From this, it follows that the Fr\"oberg conjecture is true for $a_1 = \cdots = a_{n+1} = 2$ and $a_{n+2} = k$, where
$$
\begin{cases}
k \geq \frac{n-2}{2} & \text{for } n \text{ even}, \\
k \geq \frac{n+1}{2} & \text{for } n \text{ odd}.
\end{cases}
$$
The statement is trivial when $n$ is odd, as in this case, $k$ is the socle degree or higher. However, for even $n$, the case $k = \frac{n-2}{2}$ offers a new result for the conjecture.

\section{A minimal free resolution of the initial ideal}
The initial ideal of $I_{n,k}$ belongs to the class of squarefree stable ideals plus powers, as studied in \cite{Murai2008}. The structure of the minimal free resolution of such ideals and formulas for computing their Betti numbers are well-established, often relying on subsets of the generators. In this section, we provide an explicit description of the Betti numbers of $I_{n,k}$ using the combinatorial properties of its generators. This allows us to directly extract the extremal Betti numbers, regularity, and projective dimension in terms of $n$ and $k$.

We briefly recall the notion of Betti numbers and refer to \cite{eisenbud2013commutative} for more details.
Consider the ideal $I$ in the polynomial ring $R=\kb[x_1, \ldots, x_n]$. Assume that
\[
F : 0 \to \cdots \to F_i \to F_{i-1} \to \cdots \to F_0 \to I \to 0
\]
is a minimal free resolution of $I$ (i.e., a  free resolution such that $\varphi_{i+1}(F_{i+1}) \subseteq m F_i$ for all $i \geq 0$) and $m=(x_1,\ldots,x_n)$. The $i$-th Betti number $\beta_i(I)$ of $I$ is defined as the rank of $F_i$. The $i$-th graded Betti number in degree $j$ 
denoted by $\beta_{i,j}(I)$, is the rank of the degree $j$ part of $F_i$.

Finding the Betti numbers of ideals is, in general, challenging. However, by semi-upper semicontinuity, we have the inequality $\beta_{i,j}(I) \leq \beta_{i,j}(\ini(I))$ for all $i,j \geq 0$. Thus, the Betti numbers of the initial ideal provide an upper bound for the Betti numbers of the ideal. Our goal in this section is to compute the Betti numbers of the initial ideal $\ini(I_{n,k})$, hence providing an upper bound for those of the ideal $I_{n,k}$.

\subsection{The squarefree part of 
$\ini(I_{n,k})$}
\begin{definition}
Let $I$ be a squarefree monomial ideal with the minimal generating set $G(I)$. Then $I$ is called \emph{strongly squarefree stable} if, for any monomial $ux_j \in G(I)$ and any index $i < j$ such that $x_i$ does not divide $u$, the monomial $ux_i$ is also in $I$.
\end{definition}

In the following, let $J_{n,k}$ denote the squarefree part of the initial ideal $\ini(I_{n,k})$. 

\begin{proposition}\label{prop:Strongly_stable} 
The ideal $J_{n,k}$ is strongly squarefree stable.
\end{proposition}
\begin{proof}
By Lemma~\ref{lem:LatticePathBijections}, the generators of $J_{n,k}$ correspond to $(N,E)$-lattice paths that touch the line $y = x + k$ for the first time after their last step. Let $P$ be a path corresponding to a generator $ux_j$. If $x_i$ does not divide $u$ and $i < j$, we need to show that $ux_i$ is also in $J_{n,k}$. 
In terms of the path, $ux_i$ corresponds to a new path $P'$ constructed from $P$ by changing the $j$-th step from going north to east, and the $i$-th step from going east to north. Since both $P$ and $P'$ take the same number of total steps and the same number of steps north, one of two cases must occur: 

(1) If the first time $P'$ touches the line $y = x + k$ is at the last step, then $P'$ corresponds to a monomial that is a minimal generator of $J_{n,k}$.

(2) If $P'$ touches the line $y = x + k$ at an earlier step, then the path up to that point corresponds to another generator $v$ of $J_{n,k}$ in an earlier degree. In this case, $ux_i$ is a multiple of $v$ and thus is in $J_{n,k}$, as desired.
\end{proof}

We now recall a formula for the Betti numbers of a strongly stable ideal from Gasharov, Hibi, and Peeva \cite{GasharovHibiPeeva2002}, as presented in \cite{Murai2008}. 
\begin{lemma}[\cite{GasharovHibiPeeva2002,Murai2008}]\label{lem:Murai}
   Let $G(I)$ be the set of minimal generators of a strongly squarefree stable ideal $I\subset R$. Then the Betti numbers $\beta_{p,p+s}(I)$
   are given by
\begin{equation}\label{eq:Gasharov_betti_formula}
\beta_{p,p+s}(I) =  
\sum_{\substack{u\in G(I)\\ \deg u = s}}\binom{\max u - s}{p},
\end{equation}
where $\max u$ is the largest index of a variable dividing the monomial $u\neq 1$.
\end{lemma}

In~\eqref{eq:Gasharov_betti_formula}, we used that $|\{j\in [n]:\; j<\max u \text{ and } x_ju \text{ is squarefree}\}|=\max u - \deg u$ for any squarefree monomial $u$. Combining Proposition~\ref{prop:Strongly_stable} and Lemma~\ref{lem:Murai}, we can determine the Betti~numbers of the squarefree part of $\ini(I_{n,k})$.

\begin{proposition}
The Betti numbers $\beta_{p,p+s}(J_{n,k})$ are zero for $n<2s-k$ and otherwise are given by
\[
\beta_{p,p+s}(J_{n,k}) = C_{s - k}^{k-1}\binom{s - k}{p}.
\]
\end{proposition}

\begin{proof}
By Proposition~\ref{prop:Strongly_stable}, the ideal $J_{n,k}$ is strongly squarefree stable. Hence, its Betti numbers are given by Lemma~\ref{lem:Murai}. By Theorem~\ref{thm:GBasis}, 
there are no minimal generators of the Gr\"obner basis of degree $s$ if $n < 2s - k$, and so the Betti numbers are zero for these values of $n$. Otherwise, for any minimal generator $u$ of degree $s$, we have that $\max u = 2s - k$. Moreover, by Corollary~\ref{cor:ConvolutionCatalanNumbers}, there are $C_{s-k}^{k-1}$ minimal generators of degree $s$ in $J_{n,k}$. Thus, the result follows by applying \eqref{eq:Gasharov_betti_formula}.
\end{proof}

We now apply Proposition~2.1 in \cite{Murai2008} to compute the Betti numbers of the ideal $\ini(I_{n,k})=J_{n,k}+(x_1^2,\ldots,x_n^2)$. Before presenting the formula, however, we need to fix the following notation.
For each squarefree monomial $u$ of $\ini(I_{n,k})$, we define
\[
A_p(u) = \sum_{j=1}^{t} \binom{i_j - 1}{p - j},\quad \text{if } u=x_{i_1}\cdots x_{i_t} \text{ with } i_1>\cdots > i_t.
\]
\begin{corollary}\label{cor:MuraiFormula}
The Betti numbers  
$\beta_{p,p+s}(\ini(I_{n,k}))$  
are given by
\begin{equation}\label{eq:Murai_formula}
\begin{split}
&\sum_{\substack{u\in \mathrm{SFP}(J_{n,k})\\ \deg(u)=s}} A_{p+1}(u)
-\sum_{\substack{u\in \mathrm{SFP}(J_{n,k})\\ \deg(u)=s-1}}
\biggl(\binom{n}{p+1}-A_{p+2}(u)\biggr)
+\delta_{s-1,p+1}\binom{n}{p+1},
\end{split}
\end{equation}
where $\delta_{s-1, p+1}$ equals $1$ if $s - 1 = p + 1$ and $0$ otherwise.
\end{corollary}

\begin{example}\label{ex:Murai_formula}
Consider the ideal $\ini(I_{4,2})=(x_1 x_2, x_1 x_3 x_4, x_2 x_3 x_4)+(x_1^2,\ldots,x_4^2)$ and its squarefree part $J_{4,2}=(x_1 x_2, x_1 x_3 x_4, x_2 x_3 x_4)$.
We compile the table 
$$
\begin{array}{r||c|c|c|c|}
u & A_1(u) & A_2(u) & A_3(u) & A_4(u) \\
\hline 
x_1 x_2         & 1 & 2 & 0  & 0 \\
\hline 
x_1 x_2 x_3     & 1 & 3 & 3  & 0 \\
\hline 
x_1 x_2 x_4     & 1 & 4 & 5  & 1 \\
\hline 
x_1 x_3 x_4     & 1 & 4 & 6  & 2 \\
\hline 
x_2 x_3 x_4     & 1 & 4 & 6  & 3 \\
\hline 
x_1 x_2 x_3 x_4 & 1 & 4 & 6  & 4 \\
\end{array}
$$
of the nonzero values of $A_p(u)$.
Now we apply Corollary~\ref{cor:MuraiFormula} to obtain the Betti numbers of $\ini(I_{4,2})$,  
gathered in the Betti diagram
$$
\begin{array}{c|cccc}
      & 0 & 1 & 2 & 3\, \\
    \hline 
    2 & 5 & 2 & 0 & 0\, \\
    3 & 2 & 15 & 16 & 5\, \\
    4 & 0 & 0 & 0 & 0\, \\
    5 & 0 & 0 & 0 & 0, \\
\end{array}
$$
where $\beta_{p,p+s}$ is the entry in row $s$ and column $p$.
\end{example}

\subsection{Mapping cone resolution}
Herzog and Takayama introduced a general mapping cone procedure in~\cite{herzog2002resolutions} for constructing free resolutions of monomial ideals. 
Mayer-Vietoris trees offer a combinatorial method to compute the support of these resolutions~\cite{Saenz-De-Cabezon2009MVT}. 
We now provide the minimal resolution of $J_{n,k}$ using~mapping~cones. 

\smallskip

Let $I$ be a monomial ideal with a unique minimal generating set of monomials $G(I) = \{g_1, \dots, g_r\}$. Define $I_i = ( g_1, \dots, g_i )$ as the ideal generated by the first $i$ generators in a chosen ordering. For each $i$, we have the exact sequence
\begin{equation*}\label{eq:ses}
0 \longrightarrow I_{i-1} \cap ( g_i ) \stackrel{j}{\longrightarrow} I_{i-1} \oplus ( g_i )  \stackrel{l}{\longrightarrow} I_i \longrightarrow 0.
\end{equation*}
In this sequence, $g_i$ is referred to as a {\em pivot monomial}. 
Assume that free resolutions $\FF_{i}'$ and $\widetilde{\FF}_i$ are known for the ideals $I'_i=I_{i-1}$ and $\widetilde{I}_i = I_{i-1} \cap ( g_i )$, respectively. Then, a (not necessarily minimal) resolution $\FF_i$ of $I_i$ can be constructed as the mapping cone of the chain complex morphism $\psi: \widetilde{\FF}_i \longrightarrow \FF_{i}'$ that lifts the inclusion $j$. 
By applying this process recursively on $i$, we obtain a free resolution $\FF = \FF_r$ of $I$, known as an {\em iterated mapping cone resolution}.

\smallskip
A Mayer-Vietoris tree of a monomial ideal $I$, denoted by MVT$(I)$ is a binary tree that represents the sequence of ideals in the iterated mapping cone process. Each node in this tree contains an ideal, given by its minimal generating set, and is labeled with a pair $(p, d)$, where $p$ is the node's position and $d$ its dimension. The root node, labeled $(1, 0)$, contains the ideal $I = I_r$. For a node labeled $(p, d)$ containing an ideal $J$ with at least two generators, there are two child nodes: the left child, labeled $(2p, d+1)$ contains the ideal $\widetilde{J}$, and the right child, labeled $(2p + 1, d)$ contains the ideal $J'$.
Nodes at position $1$ and all even positions are referred to as the {\em relevant} nodes of the tree. 

By \cite[Theorem~1]{Saenz-De-Cabezon2009MVT}, there exists a resolution of $I$ supported on the relevant nodes of MVT$(I)$. This resolution is referred to as the {\em Mayer-Vietoris resolution} of $I$. The dimension of each node in MVT$(I)$ corresponds to the homological degree of the corresponding module in the resolution. The following proposition gives a sufficient condition for a Mayer-Vietoris resolution to be minimal.

\begin{proposition}[\cite{Saenz-De-Cabezon2009MVT}, Proposition~3]\label{prop:repeated_nodes}
Let $\mu \in \mathbb{N}^n$ be a multidegree such that $x^\mu$ is a minimal generator in exactly $k$ relevant nodes of the Mayer-Vietoris tree of $I$, and all of these nodes have the same dimension $i$. Then, $\beta_{i, \mu}(I) = k$. 
In particular, if $x^\mu$ appears as a minimal generator in only one relevant node,  
then $\beta_{i, \mu}(I) = 1$. 
\end{proposition}

\begin{remark}
In any non-minimal resolution, there is at least one reduction pair of generators in consecutive homological degrees with the same multidegree. Thus, a sufficient condition for a Mayer-Vietoris resolution to be minimal is the absence of repeated multidegrees in consecutive dimensions.
\end{remark}

\begin{example}\label{ex:MVT}
{Consider the ideal $I_{4,2} = ( x_1x_2, x_1x_3x_4,x_3x_3x_4,x_1^2,x_2^2,x_3^2,x_4^2 )$}. Figure \ref{fig:MVT} depicts a Mayer-Vietoris tree of $I_{4,2}$, in which we display the first few nodes. Each node is labeled with its position $p$, dimension $d$, and the corresponding ideal $I$, denoted as $(p, d) \, I$. The relevant nodes are shaded darker.
\begin{figure}
\begin{center}
 \begin{tikzpicture}[scale=0.6,transform shape]
 \tikzstyle{level 1}=[sibling distance=9cm]
 \tikzstyle{level 2}=[sibling distance=5cm]
 \tikzstyle{level 3}=[sibling distance=3cm] 
 \node{$(1,0)$ $x_1x_2,x_1x_3x_4,x_2x_3x_4, x_1^2, x_2^2, x_3^2,x_4^2$}
 child{ node{$(2,1)$ $x_4^2(x_1x_2, x_1x_3,x_2x_3,x_1^2, x_2^2,x_3^2)$}
 	child{node[left=.5cm,rectangle,draw]{$(4,2)$ $x_3^2x_2^2(x_1,x_2)$}}
 	child{node[left=.3cm,xshift=.2cm,color=black!50!white]{$(5,1)$ $x_4^2(x_1x_2, x_1x_3,x_2x_3,x_1^2, x_2^2)$}
 	      child{node[xshift=-1cm,rectangle,draw]{$(10,2)$ $x_2^2x_4^2(x_1,x_3)$}}
          child{node[below=.5cm,xshift=-1cm,color=black!50!white]{$(11,1)$ $x_4^2(x_1x_2,x_1x_3,x_2x_3,x_1^2)$}
          child{node[xshift=-1cm,rectangle,draw]{$(22,2)$ $x_1^2x_4^2(x_2,x_3)$}}
	child{node[below=.6cm,xshift=-1cm,color=black!50!white,rectangle,draw]{$(23,1)$ $x_4^2(x_1x_2,x_1x_3,x_2x_3)$}}}
          }}
 child{ node[color=black!50!white]{$(3,0)$ $x_1x_2,x_1x_3x_4,x_2x_3x_4, x_1^2, x_2^2, x_3^2$}
 	child{node[xshift=-.3cm]{$(6,1)$ $x_3^2(x_1x_2,x_1x_4,x_2x_4, x_1^2, x_2^2)$}
    child{node[xshift=-1cm,rectangle,draw]{$(12,2)$ $x_2^2x_3^3(x_1,x_4)$}}
	child{node[below=.5cm,xshift=-1cm,color=black!50!white]{$(13,1)$ $x_3^2(x_1x_2,x_1x_4,x_2x_4, x_1^2)$}
    child{node[xshift=-1cm,rectangle,draw]{$(26,2)$ $x_1^2x_3^2(x_2,x_4)$}}
	child{node[below=.6cm,xshift=-1cm,color=black!50!white,rectangle,draw]{$(27,1)$ $x_3^2(x_1x_2,x_1x_4,x_2x_4)$}}}}
 	child{node[color=black!50!white]{$(7,0)$ $x_1x_2,x_1x_3x_4,x_2x_3x_4, x_1^2, x_2^2$}
	child{node[xshift=-1cm,rectangle,draw]{$(14,1)$ $x_2^2(x_1,x_3x_4)$}}
	child{node[below=.5cm,xshift=-1cm,color=black!50!white]{$(15,0)$ $x_1x_2,x_1x_3x_4,x_2x_3x_4, x_1^2$}
    child{node[xshift=-1cm,rectangle,draw]{$(30,1)$ $x_1^2(x_2,x_3x_4)$}}
	child{node[below=.6cm,xshift=-1cm,color=black!50!white,rectangle,draw]{$(31,0)$ $x_1x_2,x_1x_3x_4,x_2x_3x_4$}}}
       }};
 \end{tikzpicture}
\end{center}
\caption{Partial display of the Mayer-Vietoris tree MVT$(\ini(I_{4,2}))$.}\label{fig:MVT}
\end{figure}

Since MVT$(I_{4,2})$ has no repeated multidegrees in nodes of consecutive dimensions, the corresponding resolution of $I_{4,2}$ is minimal. In particular, $\beta_0(I_{4,2}) = 7$, $\beta_1(I_{4,2}) = 17$, $\beta_2(I_{4,2}) = 16$ and $\beta_3(I_{4,2}) = 5$, with their multidegrees given by the monomials in the nodes at position $1$ for $\beta_0$, at positions $2$, $6$, $14$, $30$, $62$ and $126$ for $\beta_1$, at positions $4$, $10$, $12$, $22$, $26$, $28$ $46$, $54$, $56$, $60$, $94$ and $110$ for $\beta_2$, and at positions $8$, $20$, $24$, $44$ and $52$ for $\beta_3$.
\end{example}

We are now ready to construct a Mayer-Vietoris resolution of $\ini(I_{n,k})$. However, before proceeding, we first recall the following well-known lemmas.

\begin{lemma}\label{lem1}
Let $I \subseteq \kb[x_1, \dots, x_n]$ be a squarefree stable monomial ideal, and let $\sigma = \{x_{i_1}, \dots, x_{i_k}\}$ be a subset of the variables. The ideal $I_\sigma = ( g_\sigma \mid g \in G(I) ) \subseteq \kb[\{x_1, \dots, x_n\} \setminus \sigma]$ is squarefree stable, where $g_\sigma$ is the monomial obtained by setting each variable in $\{x_1, \dots, x_n\} \setminus \sigma$ equal to $1$ in the monomial $g$.
\end{lemma}

\begin{lemma}\label{lem2}
Let $I \subseteq \kb[x_1, \dots, x_n]$ be a monomial ideal, and let $y^\mu \in \kb[y_1, \dots, y_m]$ be a monomial in a different set of variables. Then the minimal free resolutions of $I$ and $I' = y^\mu I = ( y^\mu g \mid g \in G(I) ) \subseteq \kb[x_1, \dots, x_n, y_1, \dots, y_m]$ are isomorphic.
\end{lemma}

\begin{lemma}[\cite{Saenz-De-Cabezon2009MVT}, Proposition 7 and Remark 7]\label{lem3}
 For every squarefree stable ideal $I$ there is a Mayer-Vietoris tree whose corresponding resolution 
 is minimal.
\end{lemma}

\noindent{\bf A Mayer-Vietoris resolution of $\ini(I_{n,k})$.} Assume that $n = k + 2m$. Let $r$ denote the number of squarefree generators of $\ini(I_{n,k})$, which, by Corollary~\ref{cor:ConvolutionCatalanNumbers}, is given by $r = \sum_{i=0}^m C_{i}^{k-1}$.
Let $J_{n,k}$ be the squarefree part of $\ini(I_{n,k})$. 

\medskip
Consider the Mayer-Vietoris tree of $\ini(I_{n,k})$, where we select as the pivot, whenever possible, the monomial corresponding to the last squared variable among the generators of the ideal that is not a common factor of all the generators. For example, at the node in position $0$, we choose $x_n^2$ as the pivot; at the node in position $2$, we choose $x_{n-1}^2$ (since $x_n^2$ is a common factor); and at the node in position $3$, we again choose $x_{n-1}^2$. The first iteration of the iterated cone procedure yields the initial elements of MVT$(\ini(I_{n,k}))$:

\begin{center}
 \begin{tikzpicture}[scale=.8,transform shape]
 \tikzstyle{level 1}=[sibling distance=8.5cm]
 \tikzstyle{level 2}=[sibling distance=6cm]
 \node{$(1,0)$ $g_1,\dots,g_r,x_1^1,x_2^2,\dots,x_n^2$}
 child{ node{$(2,1)$ $x_n^2(g_1\vert_{x_n=1},\dots, g_r\vert_{x_n=1},\dots,x_1^2,x_2^2,\dots,x_{n-1}^2)$}}
 child{ node[color=black!50!white]{$(3,0)$ $g_1,\dots,g_r,x_1^1,x_2^2,\dots,x_{n-1}^2$}};
 \end{tikzpicture}
 \end{center}
The notation $x^\mu(x^{\nu_1}, \dots, x^{\nu_k})$ above means that each monomial $x^{\nu_j}$ is multiplied by $x^\mu$, and the monomial $g_i \vert_{x_k = 1}$ is obtained by evaluating the variable $x_k$ in $g_i$ to $1$.~This process continues until there are no more common squared variables that are not common factors. At this point, we reach nodes of the form $x_{i_1}^2 \cdots x_{i_k}^2(x^{\nu_1}, \dots, x^{\nu_l})$, where each monomial $x^{\nu_j}$ is in $\kb[\{x_1, \dots, x_n\} \backslash \{x_{i_1}, \dots, x_{i_k}\}]$. Nodes whose non-common part is squarefree are called {\em final nodes}. In Example~\ref{ex:MVT}, the final nodes are those bordered at positions $4,10,12,14,22,23,26,27,30$ and $31$.
\begin{proposition}
The cone resolution supported on the Mayer-Vietoris tree constructed above provides a minimal free resolution for $\ini(I_{n,k})$.
\end{proposition} 

\begin{proof}
We only need to show the minimality of the associated Mayer-Vietoris resolution. Note that there is no common multidegree among relevant nodes in the subtrees attached to non-final nodes, since the common factor differs in each subtree. Moreover, by Lemma~\ref{lem1}, all ideals in relevant nodes are of the form $x_{i_1}^2 \cdots x_{i_d}^2(J)$, where $J$ is a squarefree stable ideal in the variables $\{x_1, \dots, x_n\} \backslash \{x_{i_1}, \dots, x_{i_d}\}$. Now, the only possible repeated multidegrees in MVT$(\ini(I_{n,k}))$ arise in the subtrees attached to final nodes. By Lemmas~\ref{lem2} and~\ref{lem3}, the cone resolutions of each of these subtrees are minimal, since the root of each of these subtrees have the form of a monomial times a squarefree stable ideal in a different set of variables. Thus the resolution of $\ini(I_{n,k})$ is minimal.
\end{proof}

\begin{corollary}
The projective dimension of $\ini(I_{n,k})$ is $n - 1$, and its regularity is $k + m$.
\end{corollary}

\begin{proof}
The final nodes with the highest dimension and degree in the Mayer-Vietoris tree of $\ini(I_{n,k})$ are obtained by sequentially choosing as pivots sequences of squared variables of length $k+m-1$. The final node at the end of such a sequence has the form $x_{l_1}^2 \cdots x_{l_{k+m-1}}^2(x_{l'_1}, \dots, x_{l'_{m+1}})$, with dimension $k+m-1$. From each of these nodes, a tree isomorphic to the one corresponding to the prime ideal $ (x_{l'_1}, \dots, x_{l'_{m+1}})$ hangs. This prime ideal has projective dimension $m$, so the highest dimension of any node in the tree is $ (k+m-1) + m = k + 2m - 1 = n - 1$. Since the Mayer-Vietoris tree provides a minimal free resolution of $\ini(I_{n,k})$, we conclude that ${\rm pd}(\ini(I_{n,k})) = n - 1$.

For the regularity, observe that the maximal degree of a generator of $\ini(I_{n,k})$ is $k + m$. Additionally, if there is more than one generator of $\ini(I_{n,k})$ of a given degree $d$, then for each such generator, there is another generator of the same degree that differs only in one variable. When taking a squared variable as a pivot, the maximal degree of the generators and the dimension in a node $J$ remains unchanged in the subsequent node $J'$. In the reduced node $\widetilde{J}$, where the dimension increases by one, a common factor of degree $2$ is added, while the maximal degree of a squarefree generator decreases by at least $1$. Therefore, the regularity does not increase. In the final nodes, the squarefree part ensures that for the generator with the highest degree, there is always another generator (of the same degree or lower) that differs only in one variable. Thus, regularity does not increase in the trees hanging from the final nodes either. Consequently, we have ${\rm reg}(\ini(I_{n,k})) = k + m$, and this value is attained in the generating set of $\ini(I_{n,k})$.
\end{proof}

\subsection{Relation between Betti numbers of $\ini(I_{n,k})$ and $\ini(I_{n-1,k})$} 
We now show that the study of $\ini(I_{n,k})$ for general $n$ can be reduced to the cases where $n = k + 2m$ for some integer $m$. Hence, in what follows, we focus on these cases.
\begin{lemma} Let $n=k+2m+1$ for some $m$. Then
\[
\beta_{i,j}(\ini(I_{n,k}))=\beta_{i,j}(\ini(I_{n-1,k}))+ \beta_{i-1,j-2}(\ini(I_{n-1,k})).
\]
\end{lemma}

\begin{proof}
    Note that $\ini(I_{n,k}) = \ini(I_{n-1,k}) + ( x_{n}^2 )$, where the ideals $\ini(I_{n-1,k})$ and $( x_{n}^2 )$ belong to  polynomial rings with disjoint sets of variables. Consequently, the minimal free resolution of $\ini(I_{n,k})$ is given by the tensor product of the minimal free resolution of $\ini(I_{n-1,k})$ and that of $( x_{n}^2 )$. Moreover, the only non-zero Betti numbers of $( x_{n}^2 )$ are $\beta_{0,0} = 1$ and $\beta_{1,2} = 1$. 
Given these facts, we derive the relation 
\begin{eqnarray*}
    \beta_{i,j}(\ini(I_{n,k})) &=& \beta_{i,j}(\ini(I_{n-1,k}))\beta_{0,0}\big( ( x_{n}^2 ) \big) + \beta_{i-1,j-2}(\ini(I_{n-1,k}))\beta_{1,2}\big( ( x_{n}^2 ) \big) \\ &=& \beta_{i,j}(\ini(I_{n-1,k}))+ \beta_{i-1,j-2}(\ini(I_{n,k})).
\end{eqnarray*}
Note that the second summand is zero when $i = 0$ or $j < 2$.
\end{proof}
\subsection{The shape of the Betti table of $\ini(I_{n,k})$}

The analysis of the structure of $\mathrm{MVT}(\ini(I_{n,k}))$ provides insight into the structure of the Betti table of $\ini(I_{n,k})$.
The Betti table of $\ini(I_{n,k})$ encodes the Betti numbers, where the entry in column $i$ and row $j$ corresponds to $\beta_{i,i+j}(\ini(I_{n,k}))$. In the table presented below, single dots represent zero values, while asterisks denote non-zero values.

\begin{proposition}
Let $n=k+2m$. Let $\beta_{i,j}$ denote $\beta_{i,j}(\ini(I_{n,k}))$. Then
\begin{enumerate}
    \item $\beta_{0,2}=n+1$ if $k=2$, $\beta_{0,2}=n$ if $k>2$, and \\ $\beta_{0,j}=C_{j-k}^{k-1}$ for all ${k}\leq j\leq {k+m}$, 
    \item $\beta_{i,2(i+1)}=\binom{n}{i+1}$ and $\beta_{i,j}=0$ for all $j<i+k$, $j\neq 2(i+1)$ for $i=1,\dots,k-3$,
    \item $\beta_{k-1+2i,k-1+2i+k+i}=C_{i+1}^{k-1}$, for $i=0,\dots,m$,
   \item for $i=0,\dots,m$, the following hold:
    \begin{itemize}
               \item $\beta_{k-1+2i,k-1+2i+j}=0$ for all $j<k+i$,        
       \item $\beta_{k-1+2i,k-1+2i+j}\neq 0$ for all $j>k+i$,
 \item 
 {$\beta_{j,j+(k+i)}=0$} for all 
 {$j>k-1+2i$},
       \item 
       {$\beta_{j,j+(k+i)}\neq 0$} for all 
       {$j<k-1+2i$}.
    \end{itemize}

\end{enumerate}
In particular, the Betti diagram of $\ini(I_{n,k})$ is 
\[
\begin{array}{c|ccccccccccc}
          & 0 & 1 & \cdots& k-1 & k & k+1 & \cdots & k-1+2i &\cdots &k-1+2m\\
      \hline
      2  & n & \cdot &\cdot&\cdot&\cdot&\cdot&\cdot&\cdot&\cdot& \cdot\\
      3  & \cdot & * &\cdot&\cdot&\cdot&\cdot&\cdot&\cdot&\cdot& \cdot\\
      4  & \cdot & \cdot &*&\cdot&\cdot&\cdot&\cdot&\cdot&\cdot& \cdot\\
      \vdots  &\cdot&\cdot&\cdot&\cdot&\cdot&\cdot&\cdot&\cdot&\cdot& \cdot\\
      k  & C_0^{k-1} & * & \cdots & C_1^{k-1} &\cdot&\cdot&\cdot&\cdot&\cdot& \cdot\\
      k+1  & C_1^{k-1} & * & * & * & * & C_2^{k-1} &\cdot&\cdot&\cdot& \cdot\\
      \vdots  &\vdots&*&*&*&*&*&\ddots&\cdot&\cdot& \cdot\\
    k+i  & C_i^{k-1} &*&*&*&*&*&\cdots&C_{i+1}^{k-1}&\cdot& \cdot\\
      \vdots  &\vdots&\vdots&\vdots&\vdots&\vdots&\vdots&\vdots&\vdots&\ddots& \cdot\\
      k+m  & C_m^{k-1} &*&*&*&*&*& * &*&*& C_{m+1}^{k-1}
      \end{array}
\]
\end{proposition}
\begin{proof}
(1) The formula follows directly from Corollary~\ref{cor:ConvolutionCatalanNumbers} and the fact that the squares of the variables are among the minimal generators of $\ini(I_{n,k})$.
\smallskip

(2) The only generators in dimension $i$ for $i = 0, \dots, k-3$ in degree $2(i+1)$ are formed by products of $i+1$ squared variables. For each $i$, there are ${{n} \choose {i+1}}$ such products, as desired. 
\smallskip 

(3) Let $i \in\{0,\dots,m\}$. Take any sequence $\sigma$ consisting on $k+i-1$ variables among the first $k+2i$ variables, such that their product (which we will also denote by $\sigma$) is not in the ideal ${\rm in}(I_{n,k})$. Then, the final node corresponding to taking as pivots the elements of $\sigma$ in descending order has the form
\[
\sigma^2(x_{l_1},\dots,x_{l_{i+1}}, \xb^{\mu_1},\dots,\xb^{\mu_s}),
\]
where $\{x_{l_1},\dots,x_{l_{i+1}}\}=\{x_1\dots,x_{k+2i}\}\setminus \sigma$, and the $\xb^{\mu_j}$ are monomials of degree strictly larger than $1$, involving also variables larger than $x_{k+2i}$.

This final node has dimension $\vert\sigma\vert=k+i-1$, and has a descendant $\sigma^2(x_{l_1},\dots,x_{l_{i+1}})$ of the same dimension, which is also a final node. The subtree hanging from this final node has a leaf such that its only generator is $\sigma^2x_{l_1}\cdots x_{l_{i+1}}$ at dimension $(k+i-1)+i=k+2i-1$, which has degree $2(k+i-1)+(i+1)$ and is a generator of the minimal free resolution of ${\rm in}(I_{n,k})$, hence contributing in one unit to $\beta_{k-1+2i,k-1+2i+k+i}$. The number of such sequences is $C^{k-1}_{i+1}$ by Corollary~\ref{cor:ConvolutionCatalanNumbers}; as a byproduct, we obtain the identity
\[
{{k+2i}\choose{k+i-1}}-\sum_{j=0}^{i-1}C^{k-1}_{j}{{2i-2j}\choose{i-j-1}}=C^{k-1}_{i+1}
\]
using the unique factorization (guaranteed by squarefree stability of $\ini(I_{n,k})$) of squarefree monomials $u\in \ini(I_{n,k})\cap\mathbf{k}[x_1,\ldots,x_{k+2i}]$ as $u=v\cdot w$, where $v$ is a minimal generator of $\ini(I_{n,k})\cap\mathbf{k}[x_1,\ldots,x_{k+2i}]$ and $\max v < \min w$.

(4) 
We show that no other sequence of variables produces a final node with a hanging tree at this specific dimension and degree. We analyze the following cases:

Sequences $\sigma$ corresponding to monomials in the ideal $\ini(I_{n,k})$ do not appear in the nodes of the tree, by construction.
 
Sequences $\sigma$ such that $|\sigma| < k + i - 1$ correspond to final nodes, where fewer than $i+1$ variables appear as isolated generators. Therefore, the descendants of these final nodes produce subtrees with either a projective dimension strictly less than $i + (k + i - 1 - |\sigma|)$, or a total degree greater than $2k + 3i - 1$. Thus, none of these contributes to $\beta_{k-1+2i, k-1+2i+k+i}$.

Sequences $\sigma$ such that $|\sigma| > k + i - 1$ correspond to final nodes in dimension $|\sigma| > k + i - 1$. To reach dimension $k-1+2i$ at the correct degree, the degree increase would need to be smaller than $k-1+2i + k+i - 2|\sigma|$ in $k-1+2i - |\sigma|$ dimensional steps, meaning at least one step produces no increase. Therefore, these sequences do not contribute to $\beta_{k-1+2i, k-1+2i+k+i}$.

Regarding the claims on non-vanishing Betti numbers, we again consider the subtree hanging from $\sigma^2(x_{l_1}, \dots, x_{l_{i+1}})$ introduced in the proof of item (3). Along the path to its leaf $\sigma^2x_{l_1} \cdots x_{l_{i+1}}$, there are relevant nodes of dimensions $k+i-1+\ell$ where $\ell = 0, \dots, i-1$, containing at least the generator $\sigma^2x_{l_1} \cdots x_{l_{\ell+1}}$, which implies that $\beta_{k+i-1+\ell, k+i-1+\ell+(k+i)} \neq 0$. To show that $\beta_{k+i-1-\ell, k+i-1-\ell+(k+i)} \neq 0$ also holds for $\ell = 1, \dots, k+i-1$, we specialize to $\sigma$ with $\sigma^2 = x_{i+2}^2 \cdots x_{k+2i}^2$. An inspection of the associated lattice paths shows that this is a valid sequence, and that $\big(\sigma^2 / x_{i+2}^2 \cdots x_{i+1+\ell}^2\big)x_{i+1} \cdots x_{k+2i}$ is a generator of degree $\ell+1+2(k+i-1-\ell) = k+i-1-\ell+(k+i)$ in a relevant node of dimension $k+i-1-\ell$. Hence, $\beta_{k+i-1-\ell, k+i-1-\ell+(k+i)} \neq 0$.

Finally, the Betti numbers data determine the shape of the Betti diagram.
\end{proof}

\section{Further directions}
We conclude the paper by presenting two conjectures and discussing potential extensions of our work to ideals of the general form $$I_{a_1,\ldots,a_{n+1}}=(x_1^{a_1}, \ldots, x_n^{a_n}, (x_1 + \cdots + x_n)^{a_{n+1}})$$ for various values of $a_i$. 
The first conjecture is on the coefficients of the Gr\"obner basis for $I_{a_1,\ldots,a_{n+1}}$ and comes from experiments in Macaulay2~\cite{M2}. In the following, we fix the reverse lexicographic order.
\begin{conjecture}
Each coefficient occurring in the Gr\"obner basis of $I_{a_1,\ldots,a_{n+1}}$ 
is a (possibly empty) product of primes less than $\max \{a_1,\ldots,a_n \}.$
\end{conjecture}

In the case of $I_{n,k}$, we have $a_1=\cdots = a_n=2$ and $a_{n+1}=k$, so the conjecture states that all nonzero coefficients should be $1$, which holds by Theorem~\ref{thm:GBasis}. 

\medskip
Given a sequence of integers $(a_1, \ldots, a_n)$, there exists an extension of square-free strongly stable ideals, called $(a_1, \ldots, a_n)$-strongly stable ideals; see \cite[p.~1322]{Murai2008}. Let $J_{a_1, \ldots, a_{n+1}}$ denote the ideal generated by all monomials $x_1^{b_1} \cdots x_n^{b_n}$ in $I_{a_1, \ldots, a_{n+1}}$ such that $b_1 < \min\{a_1, a_{n+1}\}$ and $b_j < a_j$ for all $j > 1$.

Our computations suggest that $J_{a_1, \ldots, a_{n+1}}$ is $(\min\{a_1, a_{n+1} \}, a_2, \dots, a_n)$-strongly stable and that the initial ideal of $I_{a_1, \dots, a_{n+1}}$ has interesting combinatorial properties. For example, we propose the following conjectural generalization of Proposition~\ref{prop:Strongly_stable}.

\begin{conjecture}
The initial ideal $\mathrm{in}(I_{a_1,\ldots,a_{n+1}})$ 
can be decomposed as 
\[
\mathrm{in}(I_{a_1,\ldots,a_{n+1}}) = (x_1^{\min\{a_1, a_{n+1} \}}, x_2^{a_2},\dots, x_n^{a_n}) + J_{a_1,\ldots,a_{n+1}}.
\]
\end{conjecture}

We remark that the recent result by Booth, Singh, and Vraciu \cite{booth2024weak} supports this conjecture for the case $a_i=3$ for $i = 1,\ldots,n+1$. Furthermore, when $a_1 = \cdots = a_n = 3$ and $a_{n+1} = 2$, our computations show that the Motzkin numbers count the elements of each degree in the initial ideal of $(x_1^3, \ldots, x_n^3, (x_1 + \cdots + x_n)^2)$, with these numbers characterized by certain lattice paths.

\smallskip

We also note that Schreyer's construction~\cite{Schreyer1980} provides a minimal free resolution for the monomial ideals $\ini(I_{n,k})$. However, we have excluded this construction here, as it closely parallels the resolution obtained via the Mayer-Vietoris construction. We expect that this property extends to other ideals of the form $I_{a_1, \ldots, a_{n+1}}$.

Moreover, since the resolutions of the ideals $\ini(I_{n,k})$ are given by mapping cones, we anticipate that these resolutions will be supported on a CW-complex, as discussed in~\cite{dochtermann2014cellular}. This will be explored in future work.

\section*{Acknowledgements}

Experimental computations in Macaulay2 
have been crucial for our work, which  started during a workshop on Gr\"obner free methods held at Universidad de La Rioja (Logro\~no, Spain). We thank the local organizers for their hospitality and Victor Ufnarovski for fruitful discussions in the beginning of the project. We also thank KU Leuven for hosting us for one week during the middle part of the project. We furthermore thank Mats Boij for providing insightful comments during talks given by the first two authors in the latter part of the project, and Per Alexandersson who pointed out the reference \cite{HAGLUND2018851} to us.

Samuel Lundqvist was supported by the Swedish Research Council grant VR2022-04009. 
Fatemeh Mohammadi was partially supported by the grants G0F5921N (Odysseus programme) and G023721N from the Research Foundation - Flanders (FWO), the UiT Aurora project MASCOT and the KU~Leuven~iBOF/23/064~grant. Eduardo S\'aenz-de-Cabez\'on was supported by grants PID2024-157733NBI00 funded by MCIN/AEI/10.13039/501100011033/FEDER EU and INICIA2023/02 from Gobierno de La Rioja (Spain).
\bibliographystyle{abbrv}
\bibliography{sample-main.bib}

@article{boijlundqvist,
 title={A classification of the weak {L}efschetz property for
almost complete intersections generated by
uniform powers of general linear forms},
  author={Boij, Mats and Lundqvist, Samuel},
  journal={Algebra \& Number Theory},
  volume={17},
number ={111--126},
  year={2023},
  publisher={MSP}
}

@article{Conca,
author = {Aldo Conca},
title = {Reduction numbers and initial ideals},
journal = {Proceedings of the American Mathematical Society},
volume = {131},
pages = {1015--1020},
year = {2003},
}

@article{Wiebe,
author = {Attila Wiebe},
title = {The {L}efschetz Property for Componentwise Linear Ideals and {G}otzmann Ideals},
journal = {Communications in Algebra},
volume = {32},
number = {12},
pages = {4601--4611},
year = {2004},
publisher = {Taylor \& Francis},
doi = {10.1081/AGB-200036809}
}

@article{LUNDQVIST2019213,
title = {On the structure of monomial complete intersections in positive characteristic},
journal = {Journal of Algebra},
volume = {521},
pages = {213-234},
year = {2019},
issn = {0021-8693},
doi = {https://doi.org/10.1016/j.jalgebra.2018.11.024},
url = {https://www.sciencedirect.com/science/article/pii/S0021869318306811},
author = {Samuel Lundqvist and Lisa Nicklasson},
}

@article{WATANABE1989194,
title = {The {D}ilworth number of {A}rtin {G}orenstein rings},
journal = {Advances in Mathematics},
volume = {76},
number = {2},
pages = {194-199},
year = {1989},
issn = {0001-8708},
doi = {https://doi.org/10.1016/0001-8708(89)90049-2},
url = {https://www.sciencedirect.com/science/article/pii/0001870889900492},
author = {Junzo Watanabe}
}

@book{Catalan_Stanley,
    author = {Stanley, Richard P.},
    title = {Catalan Numbers},
    publisher = {Cambridge University Press},
    year = {2015}
}

@article{Stanley,
author = {Stanley, Richard P.},
title = {Weyl Groups, the Hard {L}efschetz Theorem, and the {S}perner Property},
journal = {SIAM Journal on Algebraic Discrete Methods},
volume = {1},
number = {2},
pages = {168-184},
year = {1980},
doi = {10.1137/0601021},

URL = { 
    
        https://doi.org/10.1137/0601021
    
    

},
eprint = { 
    
        https://doi.org/10.1137/0601021
    

}
}

@article{Watanabe_Boolean,
title={{The determinants of certain matrices arising from the Boolean lattice}}, 
author={Masao Hara and Junzo Watanabe},
journal = {Discrete Mathematics},
year = {2008},
volume = {308},
pages = {5815-5822},
number = {23}
}

@article{Froberg,
 Author = {Fr{\"o}berg, Ralf},
 Title = {An inequality for {Hilbert} series of graded algebras},
 FJournal = {Mathematica Scandinavica},
 Journal = {Mathematica Scandinavica},
 ISSN = {0025-5521},
 Volume = {56},
 Pages = {117--144},
 Year = {1985},
 Language = {English},
 DOI = {10.7146/math.scand.a-12092},
 zbMATH = {3931156},
 Zbl = {0582.13007}
}

@article{Hollman,
    author = {Fröberg, Ralf and Hollman, Joachim},
    title = {Hilbert Series for Ideals Generated by Generic Forms},
    journal = {Journal of Symbolic Computation},
    year = {1994},
    Volume = {17},
    Pages = {149-157}
}

@article{Tedford,
    author = {Tedford, Steven},
    title = {Combinatorial interpretations of convolutions of the {Catalan} numbers},
    journal = {Integers},
    year = {2011},
    Volume = {11},
    Pages = {35-45}
}

@article{herzog2002resolutions,
  title={Resolutions by mapping cones},
  author={Herzog, J{\"u}rgen and Takayama, Yukihide},
  journal={Homology, Homotopy and Applications},
  volume={4},
  number={2},
  pages={277--294},
  year={2002}
}

@article{dochtermann2014cellular,
  title={Cellular resolutions from mapping cones},
  author={Dochtermann, Anton and Mohammadi, Fatemeh},
  journal={Journal of Combinatorial Theory, Series A},
  volume={128},
  pages={180--206},
  year={2014},
  publisher={Elsevier}
}

@ARTICLE{Saenz-De-Cabezon2009MVT,
	author = {Sáenz-De-Cabezón, Eduardo},
	title = {Multigraded {B}etti numbers without computing minimal free resolutions},
	year = {2009},
	journal = {Applicable Algebra in Engineering, Communications and Computing},
	volume = {20},
	number = {5-6},
	pages = {481 – 495},
	doi = {10.1007/s00200-009-0112-6},
	type = {Article}
	
	}

@book{eisenbud2013commutative,
  title={Commutative algebra: with a view toward algebraic geometry},
  author={Eisenbud, David},
  volume={150},
  year={2013},
  publisher={Springer Science \& Business Media}
}

@Article{GasharovHibiPeeva2002,
 Author = {Gasharov, Vesselin and Hibi, Takayuki and Peeva, Irena},
 Title = {Resolutions of a-stable ideals.},
 FJournal = {Journal of Algebra},
 Journal = {Journal of Algebra},
 ISSN = {0021-8693},
 Volume = {254},
 Number = {2},
 Pages = {375--394},
 Year = {2002},
 Language = {English},
 DOI = {10.1016/S0021-8693(02)00083-2},
 zbMATH = {1867985},
 Zbl = {1089.13508}
}

@Mastersthesis{Schreyer1980,
Author = {Frank-Olaf {Schreyer} },
    Title = {{Die Berechnung von Syzygien mit dem verallgemeinerten Weierstrass'schen Divisionssatz.}},
    School    = {University of Hamburg, Germany},
     Year      = {1980}
}

@Article{EmsalemIarrobino,
    Author = {J. Emsalem and I. Iarrobino},
    Title = {{Inverse system of a symbolic power, I}},
    Journal = {{Journal of Algebra}},
    Volume = {174},
    Pages = {1080--1090},
    Year = {1995},
}

@Article{SturmfelsXu,
    Author = {B. Sturmfels and Z. Xu},
    Title = {{Sagbi bases of Cox–Nagata rings}},
    Journal = {Journal of the European Mathematical Society},
    Volume = {12},
    Pages = {429--459},
    Year = {2010},
}

@article{Graver,
title = {The module structure of integral designs},
journal = {Journal of Combinatorial Theory, Series A},
volume = {15},
number = {1},
pages = {75-90},
year = {1973},
issn = {0097-3165},
doi = {https://doi.org/10.1016/0097-3165(73)90037-X},
url = {https://www.sciencedirect.com/science/article/pii/009731657390037X},
author = {J.E Graver and W.B Jurkat}
}

@article{Kantor,
  title={On incidence matrices of finite projective and affine spaces},
  author={William M. Kantor},
  journal={Mathematische Zeitschrift},
  year={1972},
  volume={124},
  pages={315-318},
  url={https://api.semanticscholar.org/CorpusID:18251481}
}

@article{Murai2008,
title={Borel-plus-powers monomial ideals},
  author={Satoshi Murai},
  journal={Journal of Pure and Applied Algebra},
  year={2008},
  volume={212},
  pages={1321-1336},
}

@article{Wilson,
  title={The necessary conditions for t-designs are sufficient for something},
  author={Wilson, Richard M},
  journal={Utilitas Mathematica},
  volume={4},
  pages={207--215},
  year={1973}
}

@article{Nicklasson_SLP_two_variables,
title={{The strong Lefschetz property of monomial complete intersections in two variables}}, 
author={Lisa Nicklasson},
journal = {Collectanea Mathematica},
year = {2018},
volume = {69},
pages = {359-375}
}

@article{booth2024weak,
  title={On the weak {L}efschetz property for ideals generated by powers of general linear forms},
  author={Booth, Matthew D and Singh, Pankaj and Vraciu, Adela},
  journal={Journal of Commutative Algebra, to appear},
  year={}
}

@article{Miglioreetal,
author = {Juan C. Migliore and Rosa M. Mir{\'o}-Roig and Uwe Nagel},
title = {{On the weak {L}efschetz property for powers of linear forms}},
volume = {6},
journal = {Algebra \& Number Theory},
number = {3},
publisher = {MSP},
pages = {487 -- 526},
year = {2012},
doi = {10.2140/ant.2012.6.487},
URL = {https://doi.org/10.2140/ant.2012.6.487}
}

@article{CruzIarrobino,
title = {High-order vanishing ideals at $n+3$ points of $P^n$},
journal = {Journal of Pure and Applied Algebra},
volume = {152},
number = {1},
pages = {75-82},
year = {2000},
issn = {0022-4049},
doi = {https://doi.org/10.1016/S0022-4049(99)00125-5},
url = {https://www.sciencedirect.com/science/article/pii/S0022404999001255},
author = {Clare D'Cruz and Anthony Iarrobino}
}

@article{Phuong,
title = {A new proof of {S}tanley’s theorem on the strong {L}efschetz property},
journal = {Colloquium Mathematicum},
volume = {173},
number = {1},
pages = {1-8},
year = {2023},
author = {Ho V. N. Phuong and Quang Hoa Tran},
}

@article{Ikeda,
    author = {Hidemi Ikeda},
    title = {Results on {Dilworth} and {Rees} numbers of {Artinian} local rings},
    journal = {Japanese Journal of Mathematics},
    year = {1996},
    volume = {22},
    pages = {147-158}
}

@article{nageltrok,
    author = {Uwe Nagel and Bill Trok},
    title = {Interpolation and the weak {L}efschetz property},
    journal = {Transactions of the American Mathematical Society},
    year = {2019},
    volume = {372},
    pages = {8849-8870}
}

@article{atour,
author = {Juan C. Migliore and Uwe Nagel},
title = {{Survey Article: A tour of the weak and strong Lefschetz properties}},
volume = {5},
journal = {Journal of Commutative Algebra},
number = {3},
publisher = {Rocky Mountain Mathematics Consortium},
pages = {329 -- 358},
year = {2013},
doi = {10.1216/JCA-2013-5-3-329},
URL = {https://doi.org/10.1216/JCA-2013-5-3-329}
}

@Inbook{book,
author="Harima, Tadahito
and Maeno, Toshiaki
and Morita, Hideaki
and Numata, Yasuhide
and Wachi, Akihito
and Watanabe, Junzo",
title="Lefschetz Properties",
bookTitle="The Lefschetz Properties",
year="2013",
publisher="Springer Berlin Heidelberg",
address="Berlin, Heidelberg",
pages="97--140",
isbn="978-3-642-38206-2",
doi="10.1007/978-3-642-38206-2_3",
url="https://doi.org/10.1007/978-3-642-38206-2_3"
}

@Misc{M2,
          author = {Grayson, Daniel R. and Stillman, Michael E.},
          title = {Macaulay2, a software system for research in algebraic geometry},
          howpublished = {Available at \url{http://www2.macaulay2.com}}
        }

@article{Kustin,
author = {Andrew R. Kustin and Adela Vraciu},
title = {The Weak {L}efschetz Property for monomial complete intersections},
volume = {366},
journal = {Transactions of the American Mathematical Society},
number = {9},
pages = {4571--4601},
year = {2014}
}

@article{Miglioretrans,
author = {Juan C. Migliore and Rosa M. Miro-Roig and Uwe Nagel},
title = {Monomial ideals, almost complete intersections and the Weak Lefschetz Property},
volume = {363},
journal = {Transactions of the American Mathematical Society},
number = {1},
pages = {229--257},
year = {2010}
}

@article{HAGLUND2018851,
title = {Ordered set partitions, generalized coinvariant algebras, and the Delta Conjecture},
journal = {Advances in Mathematics},
volume = {329},
pages = {851-915},
year = {2018},
issn = {0001-8708},
doi = {https://doi.org/10.1016/j.aim.2018.01.028},
url = {https://www.sciencedirect.com/science/article/pii/S0001870816314591},
author = {James Haglund and Brendon Rhoades and Mark Shimozono}
}

@article{haglund2018delta,
  title={The delta conjecture},
  author={Haglund, James and Remmel, Jeff and Wilson, Andrew},
  journal={Transactions of the American Mathematical Society},
  volume={370},
  number={6},
  pages={4029--4057},
  year={2018}
}

@article{klmo,
 title={The Gr\"obner basis for powers of a general linear form in a monomial complete intersection},
  author={Jonsson Kling, Filip  and Lundqvist, Samuel and Mohammadi, Fatemeh and Orth, Matthias},
  journal={Transactions of the American Mathematical Society, to appear},
    year = {},
}

\end{document}